\documentclass{amsart}
\usepackage{amssymb}
\usepackage{amsthm}
\usepackage{amsmath}
\usepackage{amsfonts}
\usepackage{bbm}

\newtheorem{theorem}{Theorem}[section]
\theoremstyle{plain}

\newtheorem{corollary}[theorem]{Corollary}

\newtheorem{lemma}[theorem]{Lemma}

\newtheorem{proposition}[theorem]{Proposition}
\newtheorem{remark}[theorem]{Remark}

\numberwithin{equation}{section}

\begin{document}
\title[Diophantine analysis under continuum many bases]{Diophantine analysis of the expansions of a fixed point under continuum many bases}
\author[F. L\"{u}]{Fan L\"{u}}
\address{Department of Mathematics\\
Sichuan Normal University\\
610066, Chengdu, P. R. China}
\email{lvfan1123@163.com}
\author[B. Wang]{Baowei Wang}
\address{School of Mathematics and Statistics\\
Huazhong University of Science and Technology\\
430074, Wuhan, P. R. China}
\email{bwei\_wang@hust.edu.cn}\email{jun.wu@hust.edu.cn}
\author[J. Wu]{Jun Wu}
\subjclass[2010]{Primary 11K55; Secondary 11J83}
\keywords{Diophantine analysis, Metric theory, Beta expansion, Full cylinder}

\begin{abstract}
In this paper, we study the Diophantine properties of the orbits of a fixed point in its expansions under continuum many bases. More precisely, let $T_{\beta}$ be the beta-transformation with base
$\beta>1$, $\{x_{n}\}_{n\geq 1}$ be a sequence of real numbers in $[0,1]$ and $\varphi\colon \mathbb{N}\rightarrow (0,1]$ be
a positive function. With a detailed analysis on the distribution of {\em full cylinders} in the base space $\{\beta>1\}$, it is shown that for any given $x\in(0,1]$,
for almost all or almost no bases $\beta>1$, the orbit of $x$ under $T_{\beta}$ can $\varphi$-well approximate the sequence $\{x_{n}\}_{n\geq 1}$
according to the divergence or convergence of the series  $\sum \varphi(n)$.
This strengthens Schmeling's result significantly and complete all known results in this aspect. Moreover, the idea presented here can also be used to determine the Lebesgue measure of the set
\begin{equation*}
    \{x\in [0,1]\colon|T^{n}_{\beta}x-L(x)|<\varphi(n) \text{ for infinitely many } n\in\mathbb{N}\},
\end{equation*}
for a fixed base $\beta>1$, where $L\colon [0,1]\rightarrow[0,1]$ is a Lipschitz function. This strengthens Boshernitzan's work on quantitative recurrence properties in $\beta$-expansion.
\end{abstract}

\maketitle
\section{Introduction}

The properties of the expansions of one real number under different bases is a long standing topic in number theory. Though many great achievements have been established, still many have not been well understood (see the monograph of Bugeaud \cite{Bu} for the achievements and also lots of unsolved questions). In this paper, we focus on the Diophantine properties of the orbit of a fixed point in its expansions under a continuum many bases, i.e. under beta-transformations for bases $\beta\in (1,\infty)$.

For any $\beta>1$, the beta-transformation $T_{\beta}\colon [0,1]\rightarrow[0,1]$ is defined by
\begin{equation}\label{f1}
    T_{\beta}(x)=\beta x-\lfloor\beta x\rfloor \quad\text{ for all } x\in [0,1],
\end{equation}
where $\lfloor\cdot\rfloor$ denotes the integral part of a real number. The beta-transformation $T_{\beta}$ introduced by R\'{e}nyi \cite{Re} not only stands as a model for expanding real numbers in non-integer bases, but also provides various styles of non-Markovian symbolic dynamical systems, e.g.,
subshift of finite type, sofic system, specified system, synchronizing system etc. \cite{Bl, Sc} as $\beta$ varies.

For a fixed point $x\in (0,1]$, its orbits under beta-transformations may have completely different distributions on $[0,1]$ when $\beta$ varies, and may reflect the essential nature of the corresponding system. For example, the collection of all bases $\{\beta\in\mathbb{R}\colon \beta>1\}$ (we also call it {\em the parameter space} as in the literature) is classified according to the distributions of the orbits $\mathcal{O}_{\beta}:=\{T^{n}_{\beta}1\colon n\ge 1\}$ of 1 by Blanchard \cite{Bl}:

Class $C_{1}$:  $\mathcal{O}_{\beta}$ is ultimately zero.

Class $C_{2}$: $\mathcal{O}_{\beta}$ is ultimately non-zero periodic.

Class $C_{3}$: $\mathcal{O}_{\beta}$ is an infinite set but $0$ is not an accumulation point of $\mathcal{O}_{\beta}$.

Class $C_{4}$: $0$ is an accumulation point of $\mathcal{O}_{\beta}$ but $\mathcal{O}_{\beta}$ is not dense in $[0,1]$.

Class $C_{5}$: $\mathcal{O}_{\beta}$ is dense in $[0,1]$.

The Diophantine property of the orbit of a general fixed point under bases $\{\beta\in\mathbb{R}\colon \beta>1\}$ is initialed in the work of Schmeling \cite{Sc} where it was shown that
\begin{theorem}[Schmeling, \cite{Sc}] For any initial point $x\in(0,1]$, its orbit under beta-transformation is dense in $[0,1]$ for $\mathcal{L}$-almost every $\beta>1$. That is, for any  $x\in (0,1]$ and $x_{0}\in[0,1]$,
\begin{equation}\label{E:xx0=1}
    \liminf_{n\rightarrow\infty}|T^{n}_{\beta}x-x_{0}|=0 \quad\text{ for } \mathcal{L}\text{-a.e. } \beta>1,
\end{equation}
where $\mathcal{L}$ means the Lebesgue measure.
\end{theorem}

In this paper, following the work of Schmeling, we consider further the convergence speed in (\ref{E:xx0=1}). Let $\{x_{n}\}_{n\geq 1}\subset [0,1]$ be a sequence of real numbers and $\varphi\colon \mathbb{N}\rightarrow (0,1]$ be a positive function. Fix $x\in(0,1]$. In analogy with the classical Diophantine approximation \cite{Khint}, we define
\begin{equation*}
    E_{x}(\{x_{n}\},\varphi)=\{\beta>1\colon |T^{n}_{\beta}x-x_{n}|<\varphi(n) \textrm{ for infinitely many } n\in\mathbb{N}\}.
  \end{equation*}
  At the current stage, the main contributions on the size of $ E_{x}(\{x_{n}\},\varphi)$ is as follows:
  \begin{itemize}
    \item Hausdorff dimension \begin{itemize}
      \item  for $x=1$ and $x_n\equiv 0$ by Persson \& Schmeling \cite{PeSc},
      \item for $x=1$ and $x_n\equiv y\in [0,1]$ by Li et. al. \cite{LiPeWaWu},
      \item for all $x\in (0,1]$ and $\{x_{n}\}_{n\geq 1}$ by L\"{u} \& Wu \cite{LuWu} . 
    \end{itemize}
    \item Lebesgue measure \begin{itemize}
      \item for $x=1$ and $x_n\equiv 0$ by L\"{u} \& Wu \cite{LuWu2}.
    \end{itemize}
  \end{itemize}
So the Lebesgue measure of $E_{x}(\{x_{n}\},\varphi)$ for all $x\in (0,1]$ and any general sequences $x_n\equiv 0$ is wanted. It should also be remarked that by the ideas of a transference principle presented in \cite{LuWu}, the Lebesgue measure would (almost) imply the dimensional theory. So a complete Lebesgue measure would complete the metric theory of the orbit of a fixed point under the expansions for all bases $\{\beta>1\}$.


We would like to make several remarks about the differences between the special case ($x=1$ and $x_{n}\equiv 0$) and the general case:
\begin{enumerate}
\item To guarantee that the two points $T^{n}_{\beta}x$ and $x_{n}$ are close enough, a natural idea is to require that their beta-expansions (see Section $2$) have a sufficiently long common prefix  \cite{LiPeWaWu}. If $x_{n}=0$ for all $n\in\mathbb{N}$, since the beta-expansions of $0$ are the same (all digits are $0$) no matter what $\beta$ is, we only need to
consider those $\beta$ for which the beta-expansion of $T^{n}_{\beta}x$ begins with a sufficiently long string of zeros. But, in the general case, the beta-expansions of $x_{n}$ under different bases $\beta$ are different. Since $\beta$ is varying all the time, it is hard to get any information for the expansion of $x_n$. So the idea in \cite{LuWu2} cannot be used here.

\item On every cylinder $I(w)$ of order $n$ in the parameter space $\{\beta\in\mathbb{R}\colon \beta>1\}$ (see Section $4$), since the function $f(\beta):=T^{n}_{\beta}x$ is continuous and strictly increasing, one knows that $f(I(w))$ is an interval $[0,t)$ starting from $0$. So, the set $$\{\beta\in I(w)\colon |T^{n}_{\beta}x-0|<\varphi(n)\}$$ is always nonempty. However, when the target $0$ is changed to be $x_n$, we do not know how large $t$ would be. Thus, possibly
 the set
      $$\{\beta\in I(w)\colon |T^{n}_{\beta}x-x_n|<\varphi(n)\}$$
is empty. Therefore, we have to focus on {\em nice} cylinders, e.g., {\em full cylinders}. This idea is possible only when full cylinders take up at least a positive proportion among all cylinders, otherwise we neglect so much. However, the current knowledge (see $(iii)$ in Lemma \ref{L:Cful}) is not sufficient.

\item The criterion of whether a sequence $(\varepsilon_1, \varepsilon_2,\cdots)$ is the beta-expansion of $x\in(0,1]$ under some base $\beta>1$ depends heavily on the expansion of the unit $1$ under the same base (see Lemma \ref{L:admin}). When $x=1$, Lemma \ref{L:admin} provides a necessary and sufficient criterion so that we only need to compare the sequence $(\varepsilon_1, \varepsilon_2,\cdots)$ with its shifts. While for a fixed $x\in (0,1)$,
      there is by no means to get any explicit information from the expansion of $x$ to the expansion of 1, nor the inverse direction. So there is no general criterion of which sequence can be the expansion of $x$ under some base $\beta>1$.
\end{enumerate}

It seems hopeless to formulate some principles to overcome the difficulties arising in (1) and (3), so we have to go around to find some other way out. The notation {\em full cylinder} introduced by Dajani \& Kraaikamp \cite{DaKr} plays an important role in the afore mentioned works. However, the current knowledge
for full cylinders is inadequate to get the Lebesgue measure of $E_{x}(\{x_{n}\},\varphi)$. So some substantial materials have to be established.
We discover
that full cylinders take up a positive proportion among all cylinders (see Proposition \ref{P:numperf}).
We also find that there is a close relation between the cylinders in beta-expansion for a fixed $\beta$ and that in the
parameter space $\{\beta\in \mathbb{R}\colon \beta>1\}$ (see Lemma \ref{L:winu} and Lemma \ref{L:CcC}), which gives us effective control on the number of (full) cylinders in the parameter space.
These facts open up the step towards a complete characterization on the
Lebesgue measure of the set $E_{x}(\{x_{n}\},\varphi)$ for any $x\in (0,1]$ and $\{x_n\}_{n\geq 1}\subset [0,1]$. We prove that

\begin{theorem}\label{T:main}
  Let $\{x_{n}\}_{n\geq 1}$ be a sequence of points in $[0,1]$ and $\varphi\colon \mathbb{N}\rightarrow (0,1]$ be a positive function. Then for any $x\in(0,1]$, the set $ E_{x}(\{x_{n}\},\varphi)$ is of zero or full Lebesgue measure in $(1,+\infty)$ according to $\sum \varphi(n)<+\infty$ or not.
\end{theorem}

Specifying $x_n=y\in [0,1]$ for all $n\in \mathbb{N}$, and using the two dimensional Fubini's theorem on the set of $(\beta, y)$'s,
we can obtain the following corollary, which strengthens (\ref{E:xx0=1}) significantly.
\begin{corollary} Let $x\in (0,1]$. For $\mathcal{L}$-almost all $\beta>1$, the set
$$\Big\{y\in [0,1]: |T^n_{\beta}x-y|<\varphi(n) \text{ for infinitely many } n\in \mathbb{N}\Big\}
$$ is of Lebesgue measure $0$ or $1$ according to
$$
\sum_{n=1}^{\infty}\varphi(n)<\infty \text{ or } =\infty.
$$
\end{corollary}

The following set is a variant of ${E}_{x}(\{x_{n}\},\varphi)$. Let $\{l_{n}\}_{n\geq 1}$ be a sequence of non-negative real numbers. For any $x\in(0,1]$, define
\begin{equation*}
    \mathcal{E}_{x}(\{x_{n}\},\{l_{n}\})=\{\beta>1\colon |T^{n}_{\beta}x-x_{n}|<\beta^{-l_{n}} \text{ for infinitely many } n\in\mathbb{N}\}.
\end{equation*}
As a corollary of Theorem \ref{T:main}, we can determine the exact Lebesgue measure of the set $\mathcal{E}_{x}(\{x_{n}\},\{l_{n}\})$ for all $x\in(0,1]$. Let
$$\beta^{\star}=\inf\{\beta>1\colon \sum\beta^{-l_{n}}<+\infty\}=\sup\{\beta>1\colon \sum\beta^{-l_{n}}=+\infty\},$$
where we define $\inf\varnothing=+\infty$ and $\sup\varnothing=1$ for the empty set $\varnothing$.

\begin{corollary}\label{T:ln}
  Let $\{x_{n}\}_{n\geq 1}\subset [0,1]$ and $\{l_{n}\}_{n\geq 1}\subset [0,+\infty)$ be two sequences of real numbers. Then for any $x\in(0,1]$,
$$\mathcal{L}(\mathcal{E}_{x}(\{x_{n}\},\{l_{n}\}))=\beta^{\star}-1.$$
\end{corollary}

We remark that
  if $\beta^{\star}=+\infty$, the above equality means that the set $\mathcal{E}_{x}(\{x_{n}\},\{l_{n}\})$ is of full Lebesgue measure in $(1,+\infty)$.

\begin{remark}
The method used in the proof of Theorem \ref{T:main} can also be applied to deal with analogous problems in a fixed dynamical system.
\end{remark}

Fix $\beta>1$ and let $L\colon [0,1]\rightarrow[0,1]$ be a Lipschitz function. Define
\begin{equation*}
    \mathfrak{R}_{\beta}(L,\varphi)=\{x\in [0,1]\colon|T^{n}_{\beta}x-L(x)|<\varphi(n) \text{ for infinitely many } n\in\mathbb{N}\}.
\end{equation*}
The Hausdorff dimension of the set $\mathfrak{R}_{\beta}(L,\varphi)$ for $L(x)=x$ was given in \cite{TaWa}. Here, with the same idea used in the proof of Theorem \ref{T:main}, we can obtain its Lebesgue measure, which strengthens Boshernitzan's work on quantitative recurrence properties in $\beta$-expansion.
\begin{theorem}\label{T:rec}
  Let $L\colon [0,1]\rightarrow[0,1]$ be a Lipschitz function and $\varphi\colon \mathbb{N}\rightarrow (0,1]$ be a positive function. Then for any $\beta>1$,
  \begin{equation*}
    \mathcal{L}(\mathfrak{R}_{\beta}(L,\varphi))=\left\{
       \begin{array}{ll}
         0, & \hbox{if $\sum\varphi(n)<+\infty$;} \\
         1, & \hbox{if $\sum\varphi(n)=+\infty$.}
       \end{array}
     \right.
  \end{equation*}
\end{theorem}

\medskip

It should be also mentioned that the properties for a fixed $\beta$ have been well studied in the literature, see  \cite{Bl, Da, Ge, Ho, Pa, Ph, Sc, Th}, etc.

The rest of this paper is organized as follows: in the next section, we will introduce some notions and known results about beta-expansions. Section $3$ and Section $4$ are devoted to the study of the distribution properties of full cylinders in beta-dynamical systems and in the parameter space $\{\beta\in \mathbb{R}\colon \beta>1\}$, respectively. In Section $5$, we shall prove the convergent part of  Theorem \ref{T:main} with the help of the Borel-Cantelli lemma. After that, using the Chung-Erd\"{o}s inequality and Knopp's lemma, we deal with the divergent part of Theorem \ref{T:main} in Section $6$. The proofs of Corollary \ref{T:ln} and Theorem \ref{T:rec} are given in the last section.

\section{Preliminaries}
In this section, we fix some notions, terminologies and some known results about beta-expansions. We mainly focus on the properties of {\em full cylinder} introduced by Dajani \& Kraaikamp \cite{DaKr} which play fundamental roles in the metric theory of $\beta$-expansions both for a fixed $\beta$ and for the parameter space. For more details, the reader is referred to the papers of R\'enyi \cite{Re}, Parry \cite{Pa}, Schmeling \cite{Sc}, Persson \& Schmeling \cite{PeSc}, Li et al. \cite{LiPeWaWu}, Bugeaud \& Liao \cite{BuL} and L\"{u} \& Wu \cite{LuWu}.

\subsection{Notation}
Let $u=u_{1}\cdots u_{m}$ and $w=w_{1}\cdots w_{n}$ be two words of nonnegative integers with $m,n\in \mathbb{N}$ and $\xi=\xi_{1}\xi_{2}\cdots$
be a sequence of nonnegative integers. Denote the length of $u$ by $|u|:=m$. For any $1\leq k\leq m$, let $u|_{k}=u_{1}u_{2}\cdots u_{k}$.
Write $uw=u_{1}\cdots u_{m}w_{1}\cdots w_{n}$ and the sequence $u\xi=u_{1}\cdots u_{m}\xi_{1}\xi_{2}\cdots$ for the concatenations. Call $u$ a prefix of the word $w$ if $1\le m\le n$ and $w|_{m}=u$, and that $u$ a prefix of the sequence $\xi$ if $\xi_{1}\cdots \xi_{m}=u$. For any $k\geq 0$, denote $u^{k}$ for the concatenations of $k$ many $u$. Put $u^{\infty}=uu\cdots$ for the sequence consisting of infinitely many copies of $u$.

The lexicographical order $\prec$ between two sequences $\xi=\xi_{1}\xi_{2}\cdots$, $ \eta=\eta_{1}\eta_{2}\cdots$
of nonnegative integers is defined as follows: $\xi\prec\eta$
if there exists an integer $i_{0}\ge 0$ such that
$\xi_{i}=\eta_{i}$ for all $i\le i_{0}$ and $\xi_{i_{0}+1}<\eta_{i_{0}+1}$. The notion $\xi\preceq \eta$ means that $\xi\prec\eta$ or $\xi=\eta$. Moreover, the lexicographical order can be extended to words: for two words $u$, $w$, one says $u\prec w$ if $u0^{\infty}\prec w0^{\infty}$.

\subsection{Beta-expansion}
Now, we recall some basic properties of beta-expansion for a fixed base. Let $\beta>1$ be a real number.
By the algorithm (\ref{f1}), every point $x\in [0,1]$ can be expanded into a finite or infinite series as
\begin{align}\label{E:x=Tn}
   x&=\frac{\varepsilon_{1}(x,\beta)}{\beta}+\frac{\varepsilon_{2}(x,\beta)}{\beta^{2}}+\cdots
   +\frac{\varepsilon_{n}(x,\beta)+T^{n}_{\beta}x}{\beta^{n}}=\sum^{\infty}_{n=1}\frac{\varepsilon_{n}(x,\beta)}{\beta^{n}},
\end{align}
where $\varepsilon_{n}(x,\beta)=\lfloor\beta T^{n-1}_{\beta} x\rfloor$ for all $n\in \mathbb{N}$. For simplicity, we also call the sequence  $$\varepsilon(x,\beta):=\varepsilon_{1}(x,\beta)\varepsilon_{2}(x,\beta)\cdots\varepsilon_{n}(x,\beta)\cdots$$ as the beta-expansion of $x$ in base $\beta$.

We write $\varepsilon_{n}(\beta):=\varepsilon_{n}(1,\beta)$ and $\varepsilon(\beta):=\varepsilon(1,\beta)$
for the expansion of 1.
If the sequence $\varepsilon(\beta)$ ends with $0^{\infty}$, let $i_0$ be the smallest integer such that $\varepsilon_{i_0}(\beta)\ne 0$,
and define the sequence $\varepsilon^{\ast}(\beta)$ by $$\varepsilon^{\ast}(\beta)
=\Big(\varepsilon_{1}(\beta)\cdots\varepsilon_{i_{0}-1}(\beta)\varepsilon^{-}_{i_{0}}(\beta)\Big)^{\infty},$$ where $\varepsilon^{-}_{i_{0}}(\beta)=\varepsilon_{i_{0}}(\beta)-1$. Otherwise, we define the sequence $\varepsilon^{\ast}(\beta)$ to be the same with $\varepsilon(\beta)$.
The sequence $\varepsilon^{\ast}(\beta)$ is usually called the infinite beta-expansion of $1$ in base $\beta$.

The next proposition is about the properties of the sequence $\varepsilon^{\ast}(\beta)$.
\begin{proposition}\label{P:asteps}
$(i)$  For any $\beta>1$, we have $\varepsilon^{\ast}(\beta)\preceq\varepsilon(\beta)$, $\varepsilon^{\ast}_{1}(\beta)\geq 1$ and
\begin{equation*}
    \sum^{\infty}_{i=1}\frac{\varepsilon^{\ast}_{i}(\beta)}{\beta^{i}}=\sum^{\infty}_{i=1}\frac{\varepsilon_{i}(\beta)}{\beta^{i}}=1,
    \quad\sum^{\infty}_{i=n+1}\frac{\varepsilon^{\ast}_{i}(\beta)}{\beta^{i}}\leq \frac{1}{\beta^{n}}\quad\text{for all } n\in \mathbb{N}.
\end{equation*}
$(ii)$  For any $1<\beta_{1}<\beta_{2}$, we have $\varepsilon(\beta_{1})\prec\varepsilon^{\ast}(\beta_{2})$.
\end{proposition}

\subsection{Admissible sequence}

For any $n\in \mathbb{N}$, let
 \begin{equation*}
    \Sigma_{n}(\beta)=\{\varepsilon_{1}(x,\beta)\cdots\varepsilon_{n}(x,\beta)\colon x\in[0,1)\},
 \end{equation*}
i.e., the collection of all possible prefixes of length $n$ of the beta-expansion of some $x\in[0,1)$ in base $\beta$, called admissible words/sequences.
The following two lemmas present a characterization of elements in $\Sigma_{n}(\beta)$ and other basic properties due to Parry \cite{Pa} and R\'{e}nyi \cite{Re}.

\begin{lemma}[\cite{Pa}]\label{L:admin}
$(i)$ Let $\beta>1$. A sequence $\xi=\xi_{1}\xi_{2}\cdots$ of nonnegative integers is the beta-expansion of some $x\in[0,1)$ in base $\beta$ if and only if
\begin{equation*}
    \sigma^{i}\xi\prec \varepsilon^{\ast}(\beta) \quad\text{ for all } i\geq 0,
\end{equation*}
where $\sigma$ is the shift operator such that $\sigma\xi=\xi_{2}\xi_{3}\cdots$. So, for any $w\in \Sigma_n(\beta)$, the sequence $w0^{\infty}$ is the beta-expansion of some $x\in[0,1)$ in base $\beta$.

$(ii)$  A sequence $\xi=\xi_{1}\xi_{2}\cdots$ of nonnegative integers with $\xi\neq 10^{\infty}$ is the beta-expansion of $1$ in some base $\beta>1$ if and only if
  \begin{equation*}
    \sigma^{i}\xi\prec \xi \quad\text{ for all } i\geq 1.
\end{equation*}

$(iii)$ For any $x\in (0,1]$, the map $\beta\mapsto \varepsilon(x,\beta)$ is strictly increasing, i.e., $1<\beta_{1}<\beta_{2}$ if and only if
   $\varepsilon(x,\beta_{1})\prec \varepsilon(x,\beta_{2})$. The map $\beta\mapsto \varepsilon^{\ast}(\beta)$ is also strictly increasing.
\end{lemma}

\begin{lemma}[\cite{Pa,Re}]\label{L:number}
$(i)$  If $1<\beta_{1}<\beta_{2}$, then for any $n\in \mathbb{N}$, $\Sigma_{n}(\beta_{1})\subset \Sigma_{n}(\beta_{2})$.
$(ii)$ Let $\beta>1$. For any $n\in \mathbb{N}$,
\begin{equation*}\beta^{n}\leq\#\Sigma_{n}(\beta)\leq \frac{\beta^{n+1}}{\beta-1},\end{equation*}
where $\#$ denotes the cardinality of a finite set.
\end{lemma}

\subsection{Cylinders in beta-expansion}
Fix $\beta>1$. For any $n\in \mathbb{N}$ and $w\in \Sigma_{n}(\beta)$, let $$\mathfrak{I}(w)=\{x\in[0,1)\colon \varepsilon_{1}(x,\beta)\cdots\varepsilon_{n}(x,\beta)=w\}$$ and call it a cylinder of order $n$. From the algorithm of beta-expansion, one has
$$
\mathfrak{I}(w)=\left[\sum_{i=1}^n\frac{w_i}{\beta^i}, \ t\right),
$$
and its length satisfies $|\mathfrak{I}(w)|\le \beta^{-n}$. So call $\mathfrak{I}(w)$ a {\em full cylinder} of order $n$ if $|\mathfrak{I}(w)|=\beta^{-n}$.

For any $n\in \mathbb{N}$, let
\begin{equation*}
    \Xi_{n}(\beta)=\{w\in \Sigma_{n}(\beta)\colon |\mathfrak{I}(w)|=\beta^{-n}\},
\end{equation*}
i.e., the collection of all $w\in \Sigma_{n}(\beta)$ such that $\mathfrak{I}(w)$ is a full cylinder of order $n$.

The following lemma collects some properties of full cylinders.
\begin{lemma}[\cite{BuWa, FaWa,LiLi}]\label{L:full}\label{L:fulcyl}
 Let $\beta>1$ and $w=w_{1}\cdots w_{n-1}w_{n}$ be a word of nonnegative integers with $n\in \mathbb{N}$.

$(i)$ $w\in \Xi_{n}(\beta)$, if and only if
$\sigma^{i}w_{1}\cdots w_{n-1}w^{+}_{n} \preceq \varepsilon_{1}(\beta)\cdots \varepsilon_{n-i}(\beta)$ for all $0\leq i\leq n-1$, where $w^{+}_{n}=w_n+1$.

$(ii)$ $w\in\Xi_{n}(\beta)$ if and only if for any $m\in \mathbb{N}$ and $v\in\Sigma_{m}(\beta)$, one has $wv\in\Sigma_{n+m}(\beta)$.

$(iii)$ If $w\in\Xi_{n}(\beta)$, then for any $m\in \mathbb{N}$ and $v\in\Xi_{m}(\beta)$, one has $wv\in\Xi_{n+m}(\beta)$.

$(iv)$ Among every $n+1$ consecutive cylinders of order $n$, there exists at least one full cylinder.
\end{lemma}

\subsection{Cylinders in parameter space}
Fix some $x\in (0,1]$. For any $n\in \mathbb{N}$, let
\begin{equation*}
    \Omega_{n}(x)=\{\varepsilon_{1}(x,\beta)\cdots\varepsilon_{n}(x,\beta)\colon \beta>1\},
\end{equation*}
i.e., the collection of all possible prefixes of length $n$ of the beta-expansion of $x$ in some base $\beta>1$.
Different to Parry's lexicographic characterization of admissible word/sequence,
it is hard to present a general characterization of the words in $\Omega_{n}(x)$ in analogy with item $(i)$ in Lemma \ref{L:admin}.

For any  $n\in\mathbb{N}$ and $w\in\Omega_{n}(x)$, let $\underline{\beta}(w)=1$ if $w=0^{n}$; otherwise, let $\underline{\beta}(w)\ge 1$ be the unique positive solution of the equation
\begin{equation*}\label{E:1=beta}
    x=\frac{w_{1}}{\beta}+\frac{w_{2}}{\beta^{2}}+\cdots+\frac{w_{n}}{\beta^{n}}.
\end{equation*}
Since $\varepsilon_{1}(\beta)\geq 1$ for any $\beta>1$, then
\begin{equation}\label{E:beta=1or}
    \underline{\beta}(w)=1 \Longleftrightarrow  \text{``} x\in(0,1) \text{ and } w=0^{n} \text{" or ``} x=1 \text{ and } w=10^{n-1} \text{"}.
\end{equation}
For any $k\in \mathbb{N}$, denote by $w^{(k)}$ the lexicographically largest word in $\Omega_{n+k}(x)$ with $w$ as a prefix.

\begin{lemma}[\cite{LuWu}]\label{L:unovbeta}
  Let $w\in\Omega_{n}(x)$ with $n\in \mathbb{N}$. The following hold:
  \begin{enumerate}
    \item If $\underline{\beta}(w)>1$, then $\varepsilon(x,\underline{\beta}(w))=w0^{\infty}$;
    \item The limit of the sequence $\left\{\underline{\beta}\left(w^{(k)}\right)\right\}_{k\ge 1}$  exists. If denote it by $\overline{\beta}(w)$, then $\overline{\beta}(w)>\underline{\beta}(w)\geq 1$.
  \end{enumerate}
\end{lemma}

For any $n\in \mathbb{N}$ and $w\in\Omega_{n}(x)$, define $$I(w)=\{\beta>1\colon\varepsilon_{1}(x,\beta)\cdots\varepsilon_{n}(x,\beta)=w\}$$ and call it a cylinder of order $n$ in the parameter space $\{\beta\in\mathbb{R}\colon\beta>1\}$.

\begin{lemma}[\cite{LuWu}]\label{L:I(w)}
  Let $w\in\Omega_{n}(x)$ with $n\in \mathbb{N}$. \begin{itemize}\item If $\underline{\beta}(w)>1$, the cylinder $I(w)$ is a half open interval $[\underline{\beta}(w),\overline{\beta}(w))$;

  \item If $\underline{\beta}(w)=1$, the cylinder $I(w)$ is an open interval $(1,\overline{\beta}(w))$.

  \item The length $|I(w)|$ of the interval $I(w)$ satisfies $|I(w)|\leq x^{-1}\left(\overline{\beta}(w)\right)^{1-n}$.\end{itemize}
\end{lemma}
It is trivial that if $u\in\Omega_{m}(x)$ is a prefix of $w\in \Omega_{n}(x)$, one has
$$I(w)\subset I(u)\quad \text{and}  \quad\underline{\beta}(u)\leq \underline{\beta}(w)<\overline{\beta}(w)\leq \overline{\beta}(u).$$

\subsection{Full cylinders in parameter space}
For any $n\in \mathbb{N}$ and $w\in \Omega_{n}(x)$, define the function $f_{w}\colon (1,+\infty)\rightarrow [0,+\infty)$ by
\begin{equation}\label{f2}
  f_{w}(\beta)=
  \beta^{n}\left(x-\sum_{i=1}^{n}\frac{w_{i}}{\beta^{i}}\right),\quad\beta\in (1,+\infty).
\end{equation}
In fact, on the interval $I(w)$, the function $f_{w}(\beta)$ is just $T^{n}_{\beta}x$ by $(\ref{E:x=Tn})$ by viewing the latter as a function of $\beta$ since $x$ is fixed. Note that the function $f_{w}$ is continuous and strictly increasing on the interval $I(w)$, since
\begin{align}\label{E:f'(beta)}
    f'_{w}(\beta)&=\beta^{n-1}\left(nx-\sum_{i=1}^{n-1}\frac{(n-i)w_i}{\beta^i}\right)\ge x\beta^{n-1}>0.
\end{align}

Write $$
J(w):=\{T^{n}_{\beta}x\colon\beta\in I(w)\}=f_w(I(w)),
$$ which is a  subinterval of $[0,1)$ since $I(w)$ is an interval. More precisely, by Lemma \ref{L:I(w)},
\begin{itemize}
  \item if $\underline{\beta}(w)>1$, then $J(w)=[0,t)$ for some $t\in (0,1]$ with $f_{w}(\overline{\beta}(w))=t$;
  \item if $\underline{\beta}(w)=1$, then $J(w)=(x,t)$ or $(0,t)$ for $x\in (0,1)$ or $x=1$, respectively (see $(\ref{E:beta=1or})$).
\end{itemize}

In analogy with full cylinders in the beta-expansion for a fixed $\beta$, the same notion can also be defined in the parameter space. Let $w\in \Omega_{n}(x)$ with $n\in \mathbb{N}$. Call $I(w)$ a {\em full cylinder} of order $n$ in the parameter space $\{\beta\in \mathbb{R}\colon \beta>1\}$ if $J(w)=[0,1)$.

For any $n\in \mathbb{N}$, let
\begin{equation*}
    \Lambda_{n}(x)=\{w\in \Omega_{n}(x)\colon J(w)=[0,1)\},
\end{equation*}
i.e., the collection of all $w\in \Omega_{n}(x)$ such that $I(w)$ is a full cylinder of order $n$ in the parameter space $\{\beta\in \mathbb{R}\colon \beta>1\}$. Note that by the discussion above, we have $\underline{\beta}(w)>1$ for all $w\in \Lambda_{n}(x)$.

\begin{lemma}[\cite{LuWu}]\label{L:Cful}
$(i)$  For any $n\in \mathbb{N}$ and $w\in \Omega_{n}(x)$ with $\underline{\beta}(w)>1$, we have $w\in \Lambda_{n}(x)$ if and only if $f_{w}(\overline{\beta}(w))=1$, i.e.,
\begin{equation*}
    x=\frac{w_{1}}{\overline{\beta}(w)}+\frac{w_{2}}{(\overline{\beta}(w))^{2}}+\cdots
+\frac{w_{n}}{(\overline{\beta}(w))^{n}}+\frac{1}{(\overline{\beta}(w))^{n}}.
  \end{equation*}
$(ii)$  Let $w=w_{1}\cdots w_{n-1}w_{n}$ be a word of nonnegative integers with $n\ge 2$. If $w_{1}\cdots w_{n-1}w^{+}_{n}\in \Omega_{n}(x)$, then $w\in \Omega_{n}(x)$. If furthermore $\underline{\beta}(w)>1$, then $w\in \Lambda_{n}(x)$.
\\
$(iii)$  For any $n\in \mathbb{N}$, among every $n+1$ consecutive cylinders of order $n$ in the parameter space $\{\beta\in \mathbb{R}\colon \beta>1\}$, there exists at least one full cylinder.
\end{lemma}
Item $(iii)$ in Lemma \ref{L:Cful} was observed by Persson \& Schmeling \cite{PeSc}  when $x=1$.

\section{Full cylinders in beta-expansion for fixed $\beta$}

The distribution properties of full cylinders for a fixed $\beta$ have been widely used in the study of the metric
properties of beta-expansions. The item $(iv)$ in Lemma \ref{L:full}
indicates that full cylinders are well distributed, which is enough for one to estimate the Hausdorff dimensions of related sets,
 see \cite{BuWa, FaWa, ShWa}, etc.

However, when dealing with the Lebesgue measures of related sets, the help of this lemma is limited, because it only guarantees us a small collection of full cylinders.
In fact, full cylinders take up a positive proportion among all cylinders as shown below. It will in turn ensure us a large collection of full cylinders in the parameter space (Lemma \ref{L:CcC}) which is a fundamental step to the proof of our main result (see the remark (2) before Theorem \ref{T:main}).

\begin{proposition}\label{P:numperf}
  Let $\beta>1$. Suppose that $\lambda\in(0, 1)$ is a real number such that
   \begin{equation*}
\lambda-\lambda\ln \lambda<\frac{(\beta-1)^{2}}{\beta^{3}}\ {\text{and}}\ \lambda<\beta^{-1}.
\end{equation*}
Then for any $n\geq -\log_{\beta} \lambda$, we have $\#\Xi_{n}(\beta)\geq \lambda\#\Sigma_{n}(\beta)$.
\end{proposition}

\begin{proof}
(1). At first, one notes that the sequence $\{\#\Xi_{n}(\beta)\}_{n\geq 1}$ is non-decreasing. In fact, for any $n\in \mathbb{N}$ and $w\in \Xi_{n}(\beta)$, one has that $w0\in \Xi_{n+1}(\beta)$. This is because $w\varepsilon^{\ast}_{1}(\beta)\in \Sigma_{n+1}(\beta)$ by item $(ii)$ in Lemma \ref{L:full}, then
item $(i)$ in Lemma \ref{L:full} is applied, since $\varepsilon^{\ast}_{1}(\beta)\geq 1$.

(2). Next, we show that
\begin{equation}\label{E:Signum}
    \#\Sigma_{n}(\beta)-\#\Xi_{n}(\beta)\leq \#\Sigma_{n-1}(\beta)\quad\text{for all } n\ge 2.
\end{equation}
Note that $$
\Sigma_{n}(\beta)=\bigcup_{u\in \Sigma_{n-1}(\beta)}\Big\{u\varsigma \in \Sigma_{n}(\beta): 0\le \varsigma\le \lfloor \beta \rfloor\Big\}.
$$ Among the latter set for a fixed $u\in \Sigma_{n-1}(\beta)$, by item $(i)$ in Lemma \ref{L:full}, only the word $u\varepsilon_{\text{max}}$ may not be in $\Xi_{n}(\beta)$ where $\varsigma_{\text{max}}$ is the maximal digit $\varsigma$ such that $u\varsigma\in \Sigma_{n}(\beta)$. This yields (\ref{E:Signum}).

(3). By an iteration of (\ref{E:Signum}) and Lemma \ref{L:number} on $\# \Sigma_n(\beta)$, one has
\begin{align}\label{E:rn}
    \beta^{n}&\leq \#\Sigma_{n}(\beta) \leq \#\Xi_{n}(\beta)+\#\Sigma_{n-1}(\beta)\leq \sum_{i=2}^n\#\Xi_{i}(\beta)+\#\Sigma_{1}(\beta).
  \end{align}
This equality enables us to conclude that $\#\Xi_{n}(\beta)$ should not be so small compared with $\#\Sigma_{n}(\beta)$. More precisely, assume $\#\Xi_{n}(\beta)< \lambda\#\Sigma_{n}(\beta)$ for some $n\geq -\log_{\beta} \lambda$. Then
by the monotonicity of $\# \Xi_{n}(\beta)$, for any $k<n$, \begin{align*}
 \sum_{i=2}^n\#\Xi_{i}(\beta)+\#\Sigma_{1}(\beta)\le (n-k)\# \Xi_n(\beta)+\sum_{i=1}^k\# \Sigma_i(\beta)\\
 \le (n-k)\lambda \frac{\beta^{n+1}}{\beta-1}+\sum_{i=1}^k\frac{\beta^{i+1}}{\beta-1}.
 \end{align*}
 Specifying the integer $k$ such that $$0\leq n+\log_{\beta}\lambda <k\leq n+1+\log_{\beta}\lambda<n,$$  one has
\begin{align*}
(n-k)\lambda \frac{\beta^{n+1}}{\beta-1}+\sum_{i=1}^k\frac{\beta^{i+1}}{\beta-1}\le \frac{-\lambda\ln \lambda}{\ln \beta}\cdot \frac{\beta^{n+1}}{\beta-1}+\frac{\lambda\beta^{n+3}}{(\beta-1)^2},
\end{align*}
which is smaller than $\beta^n$ when $\lambda$ is sufficient small.  This contradicts (\ref{E:rn}).
\end{proof}

\section{Full cylinders in parameter space}
From now on until to the the proof of Theorem \ref{T:rec} in Section 7, let $x\in (0,1]$ be a fixed real number. In this section, we aim at upper bound on the number of cylinders (Lemma \ref{L:winu}) and lower bound on the number of full cylinders (Lemma \ref{L:CcC}) in the parameter space.

The following close link between the cylinders in a fixed beta-expansion and those in the parameter space $\{\beta\in \mathbb{R}\colon\beta>1\}$
give us those  bounds effectively. 
\begin{lemma}\label{L:winu}
 Let $w\in\Omega_{n}(x)$ with $n\in \mathbb{N}$. Then for any $\beta>\underline{\beta}(w)\geq 1$, one has $w\in \Sigma_{n}(\beta)$. In particular, if $u$ is a prefix of $w$, then $w\in \Sigma_{n}(\overline{\beta}(u))$.
\end{lemma}
\begin{proof}
If $\underline{\beta}(w)=1$, then by $(\ref{E:beta=1or})$, we have $w=0^{n}$ or $10^{n-1}$.
It is easy to see that $w\in \Sigma_{n}(\beta)$ by item $(i)$ in Lemma \ref{L:admin} on the criterion of admissibility.

If $\underline{\beta}(w)>1$, then by Lemma \ref{L:unovbeta}, we have $\varepsilon(x,\underline{\beta}(w))=w0^{\infty}$. Since $\varepsilon^{\ast}(\underline{\beta}(w))\preceq \varepsilon(\underline{\beta}(w))$, by item $(i)$ for $x\in (0,1)$ and item $(ii)$ for $x=1$ in Lemma \ref{L:admin}, we obtain that
\begin{equation*}
\sigma^{i}w0^{\infty}\preceq \varepsilon(\underline{\beta}(w))\prec \varepsilon^{\ast}(\beta)\quad \text{ for all } i\geq 0,
\end{equation*} where the last inequality follows from item $(ii)$ in Proposition \ref{P:asteps}.
Then by item $(i)$ in Lemma \ref{L:admin} again, it follows that $w\in \Sigma_n(\beta)$.

If $u$ is a prefix of $w$, then $\overline{\beta}(u)>\underline{\beta}(w)$,
 and thus $w\in \Sigma_{n}(\overline{\beta}(u))$.
\end{proof}

 The following result, together with Proposition \ref{P:numperf}, gives us a sufficiently large collection of full cylinders in the parameter space.
\begin{lemma}\label{L:CcC}
  Let $w\in \Lambda_{n}(x)$ with $n\in \mathbb{N}$. Then for any $m\in \mathbb{N}$ and  $v\in \Sigma_{m}(\underline{\beta}(w))$, we have $wv\in \Omega_{n+m}(x)$.
  Furthermore, if $v\in \Xi_{m}(\underline{\beta}(w))$, then $wv\in \Lambda_{n+m}(x)$.
\end{lemma}
\begin{proof}
(1) We prove that if $v\in \Sigma_{m}(\underline{\beta}(w))$, then $wv\in \Omega_{n+m}(x)$. Recall that for any $w\in \Lambda_{n}(x)$, we have $\underline{\beta}(w)>1$. The definition of $\underline{\beta}(w)$ says that
\begin{equation*}
    x=\frac{w_{1}}{\underline{\beta}(w)}+\cdots+\frac{w_{n}}{(\underline{\beta}(w))^{n}}.
\end{equation*}
Let $\beta_{1}$ be the unique positive solution of the equation
\begin{equation}\label{E:x=wvbeta}
    x=\frac{w_{1}}{\beta}+\cdots+\frac{w_{n}}{\beta^{n}}+\frac{v_{1}}{\beta^{n+1}}+\cdots+\frac{v_{m}}{\beta^{n+m}}.
\end{equation}
Then it is clear that $\beta_{1}\geq \underline{\beta}(w)>1$. We shall show that $\varepsilon_{1}(x,\beta_{1})\cdots\varepsilon_{n+m}(x,\beta_{1})=wv$, and thus $wv\in \Omega_{n+m}(x)$.

Since $v\in \Sigma_{m}(\underline{\beta}(w))$ and $\beta_{1}\geq \underline{\beta}(w)$, one has $v\in \Sigma_{m}(\beta_{1})$ by Lemma \ref{L:number}.
Then, by item $(i)$ in Lemma \ref{L:admin}, the sequence $v0^{\infty}$ is the beta-expansion of some $y\in[0,1)$ in base $\beta_{1}$. Thus, by $(\ref{E:x=wvbeta})$, we have
\begin{equation*}
    x=\frac{w_{1}}{\beta_{1}}+\cdots+\frac{w_{n}}{\beta_{1}^{n}}+\frac{y}{\beta_{1}^{n}}<\frac{w_{1}}{\beta_{1}}
    +\cdots+\frac{w_{n}}{\beta_{1}^{n}}+\frac{1}{\beta_{1}^{n}}.
\end{equation*}
Hence, by item $(i)$ in Lemma \ref{L:Cful} on the definition of $\overline{\beta}(w)$ for $w\in \Lambda_{n}(x)$, we obtain $\beta_{1}<\overline{\beta}(w)$.

Combining the above bounds on $\beta_1$ together, i.e.
$$\underline{\beta}(w)\le \beta_{1}<\overline{\beta}(w),$$
 one has $\beta_{1}\in I(w)$. Therefore, $\varepsilon_{1}(x,\beta_{1})\cdots\varepsilon_{n}(x,\beta_{1})=w$.

Since $\beta_{1}\in I(w)$, by $(\ref{E:x=Tn})$ it follows that
\begin{equation*}
  T^{n}_{\beta_{1}}x=\beta_{1}^{n}\left(x-\sum^{n}_{i=1}\frac{w_{i}}{\beta_{1}^{i}}\right)=y.
\end{equation*}
This leads to that for all $1\leq i\leq m$,
\begin{equation*}
  \varepsilon_{n+i}(x,\beta_{1})=\lfloor \beta_{1} T^{n+i-1}_{\beta_{1}}x\rfloor=\lfloor \beta_{1} T^{i-1}_{\beta_{1}}(T^{n}_{\beta_{1}}x)\rfloor=\lfloor \beta_{1} T^{i-1}_{\beta_{1}}y\rfloor=\varepsilon_{i}(y,\beta_{1})=v_{i},
\end{equation*}
i.e., $\varepsilon_{n+1}(x,\beta_{1})\cdots\varepsilon_{n+m}(x,\beta_{1})=v$. Therefore, $\varepsilon_{1}(x,\beta_{1})\cdots\varepsilon_{n+m}(x,\beta_{1})=wv$.

(2) We prove that if $v\in \Xi_{m}(\underline{\beta}(w))$, then $wv_{1}\cdots v_{m-1}v^{+}_{m}\in \Omega_{n+m}(x)$. Once this is proven, we can conclude by item $(ii)$ in Lemma \ref{L:Cful} that $wv\in \Lambda_{n+m}(x)$, since $\underline{\beta}(wv)\geq \underline{\beta}(w)>1$.

Let $\beta_{2}$ be the unique positive solution of the equation
\begin{equation*}
    x=\frac{w_{1}}{\beta}+\cdots+\frac{w_{n}}{\beta^{n}}+\frac{v_{1}}{\beta^{n+1}}+\cdots+\frac{v_{m-1}}{\beta^{n+m-1}}+\frac{v^{+}_{m}}{\beta^{n+m}}.
\end{equation*}
It is clear that $\beta_{2}>\underline{\beta}(w)>1$. By Proposition \ref{P:asteps}, we have $\varepsilon(\underline{\beta}(w))\prec \varepsilon^{\ast}(\beta_{2})$. Hence, applying item $(i)$ in Lemma \ref{L:fulcyl} to $v$, it follows that for all $0\leq i\leq m-1$,
\begin{equation*}
    \sigma^{i}v_{1}\cdots v_{m-1}v^{+}_{m}\preceq \varepsilon_{1}(\underline{\beta}(w))\cdots \varepsilon_{m-i}(\underline{\beta}(w))\preceq \varepsilon^{\ast}_{1}(\beta_{2})\cdots \varepsilon^{\ast}_{m-i}(\beta_{2}),
\end{equation*}
which implies $v_{1}\cdots v_{m-1}v^{+}_{m}\in \Sigma_{m}(\beta_{2})$ by item $(i)$ in Lemma \ref{L:admin}. Through the same process as in the proof of part (1),
we conclude that
\begin{equation*}
    \varepsilon_{1}(x,\beta_{2})\cdots\varepsilon_{n+m}(x,\beta_{2})=wv_{1}\cdots v_{m-1}v^{+}_{m}.
\end{equation*}
Therefore, $wv_{1}\cdots v_{m-1}v^{+}_{m}\in \Omega_{n+m}(x)$.
\end{proof}

The following proposition provides a lower bound on the lengths of full cylinders in the parameter space $\{\beta\in \mathbb{R}\colon \beta>1\}$.
\begin{proposition}\label{P:Cfulen}
  For any $n\in \mathbb{N}$ and $w\in \Lambda_{n}(x)$, we have
\begin{equation*}
    |I(w)|\geq (\underline{\beta}(w)-1)^{2}(\overline{\beta}(w))^{-1-n}.
\end{equation*}
\end{proposition}
\begin{proof}
  Recall that for any $w\in\Lambda_{n}(x)$, we have $\underline{\beta}(w)>1$. By Lemma \ref{L:unovbeta}, it follows that $\varepsilon(x,\underline{\beta}(w))=w0^{\infty}$. Then, by the definition of beta-expansion, we know that $w_{i}\in \{0,1,\cdots,\left\lfloor\underline{\beta}(w)\right\rfloor\}$ for all $1\leq i\leq n$. Recall the definition of $\underline{\beta}(w)$:
\begin{equation*}\label{E:x=unbeta}
    x=\frac{w_{1}}{\underline{\beta}(w)}+\cdots+\frac{w_{n}}{(\underline{\beta}(w))^{n}}
=\sum^{n}_{i=1}\frac{w_{i}}{(\underline{\beta}(w))^{i}}.
\end{equation*}
On the other hand, by Lemma \ref{L:Cful}, we have
 \begin{equation*}\label{E:x=ovbeta}
    x=\frac{w_{1}}{\overline{\beta}(w)}+\cdots
+\frac{w_{n}}{(\overline{\beta}(w))^{n}}+\frac{1}{(\overline{\beta}(w))^{n}}
=\sum^{n}_{i=1}\frac{w_{i}}{(\overline{\beta}(w))^{i}}+\frac{1}{(\overline{\beta}(w))^{n}}.
  \end{equation*}
Thus, it follows that
\begin{align*}
    \frac{1}{(\overline{\beta}(w))^{n}}&=\sum^{n}_{i=1}\frac{w_{i}}{(\underline{\beta}(w))^{i}}-\sum^{n}_{i=1}\frac{w_{i}}{(\overline{\beta}(w))^{i}}
\leq\underline{\beta}(w)\left(\sum_{i=1}^{\infty}\frac{1}{(\underline{\beta}(w))^{i}}-\sum_{i=1}^{\infty}\frac{1}{(\overline{\beta}(w))^{i}}\right)\\
&=\frac{\underline{\beta}(w)(\overline{\beta}(w)-\underline{\beta}(w))}{(\underline{\beta}(w)-1)(\overline{\beta}(w)-1)}.
\end{align*}
Therefore,
\begin{equation*}
    |I(w)|=\overline{\beta}(w)-\underline{\beta}(w)\geq\frac{(\underline{\beta}(w)-1)(\overline{\beta}(w)-1)}{\underline{\beta}(w)(\overline{\beta}(w))^{n}}\geq (\underline{\beta}(w)-1)^{2}(\overline{\beta}(w))^{-1-n}.
\end{equation*}
\end{proof}

The following lemma will be used several times in the sequel.
\begin{lemma}\label{L:0k1}
 Let $k\geq 0$ and $w\in\Omega_{n}(x)$ with $n\ge 1$. If $w0^{k}1\notin \Omega_{n+k+1}(x)$, then $I(w)=I(w0^{k+1})$, i.e., $$\varepsilon_{1}(x,\beta)\cdots \varepsilon_{n+k+1}(x,\beta)=w0^{k+1} \  {\text{for all}}\  \beta\in I(w).$$
\end{lemma}
\begin{proof}
We first prove that if there exists some $k\geq 0$ such that $w0^{k}1\notin \Omega_{n+k+1}(x)$, then we must have $\underline{\beta}(w)>1$.
Assume on the contrary that $\underline{\beta}(w)=1$. Then by $(\ref{E:beta=1or})$, one knows
\begin{align*}
w0^k1=\left\{
        \begin{array}{ll}
          0^n0^k1, & \hbox{when $x\in (0,1)$;} \\
          10^{n-1}0^k1, & \hbox{when $x=1$.}
        \end{array}
      \right.
\end{align*}Let $\beta$ be the unique positive solution of the equation
\begin{equation*}
    x=\frac{w_{1}}{\beta}+\cdots+\frac{w_{n}}{\beta^{n}}+\frac{1}{\beta^{n+k+1}}.
\end{equation*}
It is clear that $\beta>1$.
By item $(i)$ for $x\in (0,1)$ and item $(ii)$ for $x=1$ in Lemma \ref{L:admin}, it is direct to check that the sequence $w0^{k}10^{\infty}$ is the beta-expansion of $x$ in base $\beta$, and thus $w0^{k}1\in \Omega_{n+k+1}(x)$. This contradicts that $w0^{k}1\notin \Omega_{n+k+1}(x)$. Therefore, we have $\underline{\beta}(w)>1$.

  Assume that there exists some $\beta\in I(w)$ such that
  $$\varepsilon_{1}(x,\beta)\cdots \varepsilon_{n+k+1}(x,\beta)\neq w0^{k+1},  \ {\text{i.e.,}}\ \ \varepsilon_{n+1}(x,\beta)\cdots \varepsilon_{n+k+1}(x,\beta)\neq 0^{k+1}.$$ Let $\varepsilon_{n+i_{0}}(x,\beta)$ be the first nonzero digit in $\varepsilon_{n+1}(x,\beta)\cdots \varepsilon_{n+k+1}(x,\beta)$.

  If $i_{0}=k+1$, then $\varepsilon_{n+i_{0}}(x,\beta)\ge 1$. So by item $(ii)$ in Lemma \ref{L:Cful}, we have $w0^{k}1\in \Omega_{n+k+1}(x)$, which contradicts with the assumption that $w0^{k}1\notin \Omega_{n+k+1}(x)$.

  If $1\leq i_{0}\leq k$, still by item $(ii)$ in Lemma \ref{L:Cful}, it follows that $w0^{i_{0}}\in \Lambda_{n+i_{0}}(x)$, since $\underline{\beta}(w0^{i_{0}})\geq\underline{\beta}(w)>1$.
  At the same time,
$0^{k-i_{0}}1\in \Sigma_{k-i_{0}+1}(\beta)$ for any $\beta>1$. Thus by Lemma \ref{L:CcC}, we have $w0^{k}1\in \Omega_{n+k+1}(x)$, which also contradicts that  $w0^{k}1\notin \Omega_{n+k+1}(x)$. Therefore, we get $I(w)=I(w0^{k+1})$.
\end{proof}

\section{Convergent Part of Theorem \ref{T:main}}
In this section, we prove the convergent part of Theorem \ref{T:main}, i.e. to show $$\mathcal{L}(E_{x}(\{x_{n}\}, \varphi))=0, \ {\text{when}}\ \sum_{n=1}^{\infty}\varphi(n)<+\infty.$$
As usual, we will use the convergence part of the Borel-Cantelli lemma to conclude this. However, as will be seen, the estimation is far from being trivial.

For any $n\in\mathbb{N}$, let $\textbf{E}_{n}(x)=\{\beta>1\colon |T^{n}_{\beta}x-x_{n}|<\varphi(n)\}$. Then
\begin{equation}\label{E:ExE}
    E_{x}(\{x_{n}\},\varphi)=\bigcap^{\infty}_{m=1}\bigcup^{\infty}_{n=m}\textbf{E}_{n}(x).
\end{equation}
To estimate the measure of $\textbf{E}_{n}(x)$, we divide the parameter space $(1,\infty)$ into cylinders. At first, we define a sequence $\{\beta_N\}_{N\ge 1}$ decreasing to 1.

Let $\delta_{x}=1$ if $x=1$, otherwise let $\delta_{x}=0$. For any $N\in \mathbb{N}$, let $\beta_{N}$ be the unique positive solution of the equation
\begin{equation*}
    x=\frac{\delta_{x}}{\beta}+\frac{1}{\beta^{N+2}}.
\end{equation*}
Then it is clear that $\beta_{N}>1$ and $\beta_{N}\searrow 1$ as $N\rightarrow\infty$. Since $\varepsilon^{\ast}_{1}(\beta_{N})\geq 1$, by item $(i)$ for $x\in(0,1)$ and item $(ii)$ for $x=1$ in Lemma \ref{L:admin}, it is direct to check that $\varepsilon(x,\beta_{N})=\delta_{x}0^{N}10^{\infty}$.

Let $a_{N}$ be a positive integer large enough such that $a_{N}\geq N+2$ and for all $n\geq a_{N}$, we have
\begin{equation}\label{E:N0}
    x\beta^{n}_{N}\geq 2n^{2}>\max\left\{\beta_{N},\frac{1}{\beta_{N}-1}\right\}.
\end{equation}
Let
\begin{equation*}
    \mathbb{U}_{a_{N}}=\{u\in \Omega_{a_{N}}(x)\colon u\succeq \delta_{x}0^{N}10^{a_{N}-N-2}\}.
\end{equation*}
Then we have $I(u)\subset [\beta_{N},+\infty)$ for all $u\in \mathbb{U}_{a_{N}}$ and
\begin{equation*}
    (1,+\infty)=\bigcup^{\infty}_{N=1}[\beta_{N},+\infty)=\bigcup^{\infty}_{N=1}\bigcup_{u\in \mathbb{U}_{a_{N}}}I(u).
\end{equation*}
Therefore, by $(\ref{E:ExE})$, it follows that
\begin{align}\label{E:(1,+)}
    E_{x}(\{x_{n}\},\varphi)&=E_{x}(\{x_{n}\},\varphi)\cap (1,+\infty)=\bigcup^{\infty}_{N=1}\bigcup_{u\in \mathbb{U}_{a_{N}}}\Big(E_{x}(\{x_{n}\},\varphi)\cap I(u)\Big)\\\nonumber
    &=\bigcup^{\infty}_{N=1}\bigcup_{u\in \mathbb{U}_{a_{N}}}\left(\bigcap^{\infty}_{m=1}\bigcup^{\infty}_{n=m}\left(\textbf{E}_{n}(x)\cap I(u)\right)\right).
\end{align}

So, to show $\mathcal{L}(E_{x}(\{x_{n}\}, \varphi))=0$, it suffices to show that for every $N\in \mathbb{N}$ and $u\in \mathbb{U}_{a_{N}}$, the limsup set $E_{x}(\{x_{n}\},\varphi)\cap I(u)$ is of Lebesgue measure $0$. This is done by applying the convergence part of the Borel-Cantelli lemma. So, the next
task is to estimate the measure of the set $\textbf{E}_{n}(x)\cap I(u)$ when $n\in\mathbb{N}$ is large (see Lemma \ref{L:ML1}).

From now on to the end of this section, we fix $N\in \mathbb{N}$ and $u\in \mathbb{U}_{a_{N}}$.
Before the estimation, let's give some words on the strategy:
\begin{itemize}\item

A natural attempt on the measure of $\textbf{E}_{n}(x)\cap I(u)$ is to decompose the set into the following one: $$
\textbf{E}_{n}(x)\cap I(u)=\bigcup_{v\in \Omega_n(x), v|_{a_N}=u}\textbf{E}_{n}(x)\cap I(v).
$$
As far as all $v\in \Omega_n(x)$ are concerned, one only has that (Lemma \ref{L:length} below) $$|\textbf{E}_{n}(x)\cap I(v)|\leq 2x^{-1}\varphi(n)(\underline{\beta}(v))^{1-n}.$$
We cannot get any relation between the right quantity with $|I(v)|$ nor what we will get when sum them over all $v$ since $\underline{\beta}(v)$ differs greatly.

\item So, on one hand, we need distinguish the {\em good} words and {\em bad} words to relate the measure of $\textbf{E}_{n}(x)\cap I(v)$ with that of $I(v)$; on the other hand, we divide $\Omega_n(x)$ into collections with long common prefix to ensure that $\underline{\beta}(v)$ differs not so much inside each collection. More precisely, let $a_{N}\leq q<n$ be some integer and decompose $$
\textbf{E}_{n}(x)\cap I(u)=\bigcup_{i=1}^2\bigcup_{w\in \Omega^{(i)}_q(x), w|_{a_N}=u}\ \ \bigcup_{v\in \Omega_n(x), v|_q=w}\textbf{E}_{n}(x)\cap I(v),
$$ where $\Omega^{(1)}_q(x)$ for {\em good} words and $\Omega^{(2)}_q(x)$ for {\em bad} words (see the notations ${\mathbb{P}}_{q}(u)$ and $\widehat{\mathbb{P}}_{q}(u)$ below).
\end{itemize}
\begin{lemma}\label{L:length}
  For any $n\in \mathbb{N}$ and $v\in \Omega_{n}(x)$, the set
  \begin{equation*}
    I(v;\varphi):=\textbf{\emph{E}}_{n}(x)\cap I(v)=\{\beta\in I(v)\colon |T^{n}_{\beta}x-x_{n}|<\varphi(n)\}
  \end{equation*}
 is an interval of length $|I(v;\varphi)|\leq 2x^{-1}\varphi(n)(\underline{\beta}(v))^{1-n}$.
\end{lemma}
\begin{proof}
Recall the definition of $f_v$ in (\ref{f2}) and $f_v(\beta)=T^n_{\beta}(x)$ for $\beta\in I(v)$. Then it follows that $$
  I(v;\varphi)=I(v)\cap f_v^{-1}\Big(x_n-\varphi(n), x_n+\varphi(n)\Big).
  $$ Since $f_v$ is continuous and increasing on $I(v)$ and $I(v)$ is an interval, so is $I(v;\varphi)$.

  Let $\beta_{\ast}, \beta_{\ast\ast}\in I(v;\varphi)$. By (\ref{E:f'(beta)}),
  $$
  2\varphi(n)\ge |f_v(\beta_{\ast})-f_v(\beta_{\ast\ast})|\ge x\underline{\beta}(v)^{n-1}\cdot |\beta_{\ast}-\beta_{\ast\ast}|.
  $$
   Then the length of $I(v;\varphi)$ follows.
\end{proof}

Note that for any $q\geq a_{N}$ and $w\in \Omega_{q}(x)$ with $w|_{a_{N}}=u$, we have $I(w)\subset I(u)\subset [\beta_{N},+\infty)$, so
\begin{equation}\label{E:lwlurwru}
    1<\beta_{N}\leq \underline{\beta}(u)\leq \underline{\beta}(w)<\overline{\beta}(w)\leq \overline{\beta}(u).
\end{equation}

The following proposition says that all members in $I(w)$ with $w\in \Omega_{q}(x)$ do not differ so much for large $q$.
\begin{proposition}\label{P:ov/un<3}
  For any $q\geq a_{N}$ and $w\in \Omega_{q}(x)$ with $w|_{a_{N}}=u$, we have
 \begin{equation*}
    \frac{\overline{\beta}(w)}{(\underline{\beta}(w))^{2}}<1\quad\text{and}\quad\left(\frac{\overline{\beta}(w)}{\underline{\beta}(w)}\right)^{q+1}<3.
  \end{equation*}
\end{proposition}
\begin{proof}
By Lemma \ref{L:I(w)}, the inequality (\ref{E:lwlurwru}) and at last the choice of $a_N$ (see (\ref{E:N0})), we have
\begin{align*}
    \frac{\overline{\beta}(w)}{(\underline{\beta}(w))^{2}}&=\frac{\underline{\beta}(w)+|I(w)|}{(\underline{\beta}(w))^{2}}\leq \frac{\underline{\beta}(w)+x^{-1}(\overline{\beta}(w))^{1-q}}{(\underline{\beta}(w))^{2}}\\
&\leq \frac{1}{\underline{\beta}(w)}+\frac{1}{x(\underline{\beta}(w))^{q+1}}\leq \frac{1}{\beta_{N}}+\frac{1}{x\beta_{N}^{q+1}}<\frac{1}{\beta_{N}}+\frac{\beta_{N}-1}{\beta_{N}}=1
\end{align*}
and
\begin{equation*}
    \left(\frac{\overline{\beta}(w)}{\underline{\beta}(w)}\right)^{q+1}\leq \left(1+\frac{1}{x(\underline{\beta}(w))^{q}}\right)^{q+1}
\leq \left(1+\frac{1}{x\beta_{N}^{q}}\right)^{q+1}\leq e^{\frac{q+1}{x\beta^{q}_{N}}}\leq e<3.
\end{equation*}
\end{proof}

Now we divide the words with prefix $u$ into two families: {\em good} and {\em bad}. Since the same situation will also appear in the divergence case, we define the good and bad families for a general word $\tau$. But in this section take $\tau=u$ is sufficient.

Let $m\geq a_{N}$ and $\tau\in \Omega_{m}(x)$ with $\tau|_{a_{N}}=u$. For any $q\geq |\tau|=m$, define
\begin{equation*}
    \mathbb{P}_{q}(\tau)=\{w\in \Omega_{q}(x)\colon w|_{m}=\tau \text{ and } w0^{q-1}1\in \Omega_{2q}(x)\}
\end{equation*}
and
\begin{equation*}
    \widehat{\mathbb{P}}_{q}(\tau)=\{w\in \Omega_{q}(x)\colon w|_{m}=\tau \text{ and } w0^{q-1}1\notin \Omega_{2q}(x)\}.
\end{equation*}
We call $I(w)$ a bad subinterval of order $q$ of $I(\tau)$ if $w\in \widehat{\mathbb{P}}_{q}(\tau)$.

 For any $m\geq a_{N}$ and $\tau\in \Omega_{m}(x)$ with $\tau|_{a_{N}}=u$, by Proposition \ref{P:ov/un<3}, we have $(\underline{\beta}(\tau))^{2}/\overline{\beta}(\tau)>1$. Thus we can choose a sufficiently large integer $b_{\tau}$  such that $b_{\tau}\geq |\tau|=m$ and for all $q\geq b_{\tau}$, we have
\begin{equation}\label{E:N1}
    \frac{1}{q^{2}}\cdot\left(\frac{(\underline{\beta}(\tau))^{2}}{\overline{\beta}(\tau)}\right)^{q}\geq \frac{\overline{\beta}(\tau)\underline{\beta}(\tau)}{x|I(\tau)|\left(\overline{\beta}(\tau)-1\right)}.
\end{equation}

The following proposition indicates that when $q\in \mathbb{N}$ is large, the total length of all bad subintervals of order $q$ of $I(\tau)$ is very small.
\begin{proposition}\label{P:1/n2}
  Let $m\geq a_{N}$ and $\tau\in \Omega_{m}(x)$ with $\tau|_{a_{N}}=u$. For any $q\geq b_{\tau}$, we have
\begin{equation*}
    \sum_{w\in \widehat{\mathbb{P}}_{q}(\tau)}|I(w)|\leq \frac{1}{q^{2}}|I(\tau)|.
\end{equation*}
\end{proposition}
\begin{proof}
On one hand, for any $w\in \widehat{\mathbb{P}}_{q}(\tau)$, by Lemma \ref{L:winu}, we have $w\in \Sigma_{q}(\overline{\beta}(\tau))$. This fact implies
$$\# \widehat{\mathbb{P}}_{q}(\tau)\leq \#\Sigma_{q}(\overline{\beta}(\tau))\le \frac{\left(\overline{\beta}(\tau)\right)^{q+1}}{\overline{\beta}(\tau)-1}.$$

On the other hand, for each $w\in \widehat{\mathbb{P}}_{q}(\tau)$,
by Lemma \ref{L:0k1}, one has $I(w)=I(w0^{q})$.
Note $\overline{\beta}(w0^{q})>\underline{\beta}(w0^q)=\underline{\beta}(w)\geq \underline{\beta}(\tau)$. Then by Lemma \ref{L:I(w)},
$$
|I(w)|=|I(w0^q)|\le x^{-1}(\overline{\beta}(w0^{q}))^{1-2q}\le x^{-1}(\underline{\beta}(\tau))^{1-2q}.
$$
Therefore, by $(\ref{E:N1})$ on the choice of $b_{\tau}$, we have
\begin{equation*}
    \sum_{w\in  \widehat{\mathbb{P}}_{q}(\tau)}|I(w)|\leq \frac{(\overline{\beta}(\tau))^{q+1}}{\overline{\beta}(\tau)-1}\cdot \frac{1}{x(\underline{\beta}(\tau))^{2q-1}}=
\frac{\overline{\beta}(\tau)\underline{\beta}(\tau)}{x(\overline{\beta}(\tau)-1)}\cdot\frac{(\overline{\beta}(\tau))^{q}}{(\underline{\beta}(\tau))^{2q}}
\leq  \frac{|I(\tau)|}{q^{2}}.
\end{equation*}
\end{proof}

With Proposition \ref{P:ov/un<3} and Proposition \ref{P:1/n2} in hand, we are able to give an upper bound on the Lebesgue measure of the set $\textbf{E}_{n}(x)\cap I(u)$ when $n\in\mathbb{N}$ is large.

\begin{lemma}\label{L:ML1}
 Let $N\in \mathbb{N}$ and $u\in \mathbb{U}_{a_{N}}$. For any $n> 2b_{u}$, we have
\begin{equation*}
    \mathcal{L}(\textbf{\emph{E}}_{n}(x)\cap I(u))\leq \left(C_{u}\varphi(n)+\frac{9}{n^{2}}\right)|I(u)|, \text{ where } C_{u}=\frac{18(\overline{\beta}(u))^{4}}{x(\underline{\beta}(u)-1)^{3}}>0.
\end{equation*}
\end{lemma}
\begin{proof}
Choose $q=\lfloor n/2\rfloor$,
so $q\geq b_{u}\geq |u|=a_{N}$. Decompose the set $\textbf{E}_{n}(x)\cap I(u)$ into the following:
\begin{align*}
    \textbf{E}_{n}(x)\cap I(u)&=\bigcup_{w\in \mathbb{P}_{q}(u)}(\textbf{E}_{n}(x)\cap I(w)) \cup \bigcup_{w\in \widehat{\mathbb{P}}_{q}(u)}(\textbf{E}_{n}(x)\cap I(w)),
\end{align*}
where all of the unions are disjoint. Then,
\begin{equation*}
   \mathcal{L}(\textbf{E}_{n}(x)\cap I(u))=\sum_{w\in \mathbb{P}_{q}(u)}\mathcal{L}(\textbf{E}_{n}(x)\cap I(w))+\sum_{w\in \widehat{\mathbb{P}}_{q}(u)}\mathcal{L}(\textbf{E}_{n}(x)\cap I(w)).
\end{equation*}
For the second summation, note that by Proposition \ref{P:1/n2}, we have
\begin{equation*}
    \sum_{w\in \widehat{\mathbb{P}}_{q}(u)}\mathcal{L}(\textbf{E}_{n}(x)\cap I(w))\leq \sum_{w\in \widehat{\mathbb{P}}_{q}(u)}|I(w)|\leq \frac{1}{q^{2}}|I(u)|\leq\frac{9}{n^{2}}|I(u)|.
\end{equation*}
For the first summation, we claim that for any $w\in \mathbb{P}_{q}(u)$,
\begin{equation}\label{f3}
    \mathcal{L}(\textbf{E}_{n}(x)\cap I(w))\leq C_{u}\varphi(n)|I(w)|,
\end{equation}
which will lead to
\begin{equation*}
    \sum_{w\in \mathbb{P}_{q}(u)}\mathcal{L}(\textbf{E}_{n}(x)\cap I(w))\leq C_{u}\varphi(n)\sum_{w\in \mathbb{P}_{q}(u)}|I(w)|\leq C_{u}\varphi(n)|I(u)|
\end{equation*}
as desired.

Fix an arbitrary $w\in \mathbb{P}_{q}(u)$. To show (\ref{f3}), we will bound $|I(w)|$ from below and $\mathcal{L}(\textbf{E}_{n}(x)\cap I(w))$ from above, respectively.

Let $k_{0}=\min\{k\geq 0\colon w0^{k}1\in \Omega_{q+k+1}(x)\}$. By the definition of $\mathbb{P}_{q}(u)$, it is clear that $0\leq k_{0}\leq q-1$, and thus $q+k_{0}<2q\leq n$.

We first note that for any $\beta\in I(w)$, $$I(w)=I(w0^{k_{0}}), \ {\text{i.e.}}\ \varepsilon_{1}(x,\beta)\cdots \varepsilon_{q+k_{0}}(x,\beta)=w0^{k_{0}}.$$ More precisely, if $k_{0}=0$, this is trivial. If $k_{0}\geq 1$, then the minimality of $k_{0}$ implies $w0^{k_{0}-1}1\notin \Omega_{q+k_{0}}(x)$. Thus by Lemma \ref{L:0k1}, it follows that $I(w)=I(w0^{k_{0}})$.

Next, we give a lower bound of $|I(w)|$. Since $w0^{k_{0}}1\in \Omega_{q+k_{0}+1}(x)$, by Lemma \ref{L:Cful} and $(\ref{E:lwlurwru})$, we have $w0^{k_{0}+1}\in \Lambda_{q+k_{0}+1}(x)$. Thus by Proposition \ref{P:Cfulen}, it follows that
\begin{align}\label{E:lowIw}
    |I(w)|\geq |I(w0^{k_{0}+1})|&\geq \left(\underline{\beta}(w0^{k_{0}+1})-1\right)^{2}\left(\overline{\beta}(w0^{k_{0}+1})\right)^{-1-(q+k_{0}+1)}\\\nonumber
&\geq (\underline{\beta}(w)-1)^{2}(\overline{\beta}(w))^{-2-q-k_{0}}.
\end{align}

Now, we estimate the Lebesgue measure of the set $\textbf{E}_{n}(x)\cap I(w)$ from above. Since $I(w)=I(w0^{k_{0}})$ and $q+k_{0}<n$,
one has
\begin{equation*}
    I(w)=I(w0^{k_{0}})=\bigcup_{\substack{v\in \Omega_{n}(x),\\ v|_{q+k_{0}}=w0^{k_{0}}}}I(v),
\end{equation*}
where the union is disjoint. Then
\begin{equation*}\label{E:EnIw}
    \textbf{E}_{n}(x)\cap I(w)=\bigcup_{\substack{v\in \Omega_{n}(x),\\ v|_{q+k_{0}}=w0^{k_{0}}}}(\textbf{E}_{n}(x)\cap I(v))=\bigcup_{\substack{v\in \Omega_{n}(x),\\ v|_{q+k_{0}}=w0^{k_{0}}}}I(v;\varphi).
\end{equation*}
On one hand, for any $v\in \Omega_{n}(x)$ with $v|_{q+k_{0}}=w0^{k_{0}}$, by Lemma \ref{L:winu}, we have  $v\in\Sigma_{n}\left(\overline{\beta}(w)\right)$,
so $$v_{q+k_{0}+1}\cdots v_{n}\in \Sigma_{n-q-k_{0}}\left(\overline{\beta}(w)\right).$$ Thus, by Lemma \ref{L:number},
$$\# \Big\{v\in \Omega_{n}(x): v|_{q+k_{0}}=w0^{k_{0}}\Big\}\le \# \Sigma_{n-q-k_{0}}\left(\overline{\beta}(w)\right)\le
\frac{(\overline{\beta}(w))^{n-q-k_{0}+1}}{\overline{\beta}(w)-1}.$$
On the other hand, by Lemma \ref{L:length}, for each $v\in \Omega_{n}(x)$ with $v|_{q+k_{0}}=w0^{k_{0}}$, $$
|I(v;\varphi)|\le 2x^{-1}\varphi(n)(\underline{\beta}(v))^{1-n}\le 2x^{-1}\varphi(n)(\underline{\beta}(w))^{1-n}.
$$
Then it follows that \begin{align*}
    \mathcal{L}(\textbf{E}_{n}(x)\cap I(w))&\le \frac{(\overline{\beta}(w))^{n-q-k_{0}+1}}{\underline{\beta}(w)-1}\cdot \frac{2\varphi(n)}{x(\underline{\beta}(w))^{n-1}}\\
&=\frac{2\varphi(n)(\overline{\beta}(w))^{3}\underline{\beta}(w)}{x(\underline{\beta}(w)-1)^{3}}\cdot \left(\frac{\overline{\beta}(w)}{\underline{\beta}(w)}\right)^{n}\cdot (\underline{\beta}(w)-1)^{2}(\overline{\beta}(w))^{-2-q-k_{0}}.
\end{align*}
Therefore,
by $(\ref{E:lowIw})$ and Proposition \ref{P:ov/un<3}, we have
\begin{align*}\label{E:L2}
    \mathcal{L}(\textbf{E}_{n}(x)\cap I(w))
    &\le \frac{2\varphi(n)(\overline{\beta}(w))^{3}\underline{\beta}(w)}{x(\underline{\beta}(w)-1)^{3}}\cdot
\left(\frac{\overline{\beta}(w)}{\underline{\beta}(w)}\right)^{n}\cdot |I(w)|\\\nonumber
&\leq \frac{2\varphi(n)(\overline{\beta}(u))^{4}}{x(\underline{\beta}(u)-1)^{3}}\cdot 3^2\cdot |I(w)|
\\&=C_{u}\varphi(n)|I(w)|.
\end{align*}
\end{proof}

\begin{proof}[Proof of Theorem \ref{T:main}: the convergent part]
For any $N\in \mathbb{N}$ and $u\in \mathbb{U}_{a_{N}}$, in Lemma \ref{L:ML1}, we have proved that for any $n\geq 2b_{u}$,
\begin{equation*}
    \mathcal{L}(\textbf{E}_{n}(x)\cap I(u))\leq \left(C_{u}\varphi(n)+\frac{9}{n^{2}}\right)|I(u)|.
\end{equation*}
Thus
$$\sum_{n=1}^{\infty} \varphi(n)<+\infty\Longrightarrow \sum_{n=1}^{\infty}\mathcal{L}(\textbf{E}_{n}(x)\cap I(u))<+\infty.$$
Then the convergent part of the Borel-Cantelli lemma is applied.
\end{proof}

\section{Divergent Part of Theorem \ref{T:main}}
Recall that $x\in (0,1]$ is a fixed real number. In this section, we will prove the divergent part of Theorem \ref{T:main}, that is to show that the set $E_{x}(\{x_{n}\},\varphi)$ is of full Lebesgue measure in $(1,+\infty)$ if $\sum\varphi(n)=+\infty$.

To get the measure of a limsup set from below, the following Chung-Erd\"{o}s inequality \cite{Chung} is widely used.

\begin{lemma}[Chung-Erd\"{o}s inequality, \cite{Chung}]\label{L:PZI} Let $(\Omega, \mathcal{B}, \nu)$ be a a finite measure space and $\{E_n\}_{n\ge 1}$ be a sequence of measurable sets.  If $\sum_{n\ge 1}\nu(E_n)=\infty$, then $$
\nu(\limsup_{n\to \infty}E_n)\ge \limsup_{n\to \infty}\frac{(\sum_{1\le n\le N}\nu(E_n))^2}{\sum_{1\le i\ne j\le N}\nu(E_i\cap E_j)}.
$$
%
\end{lemma}

In many applications, Chung-Erd\"{o}s inequality enables one to conclude the positiveness of $\nu(\limsup E_n)$, so to get a full measure result for $\limsup E_n$, one can apply Chung-Erd\"{o}s inequality locally, i.e. apply it to the set $\limsup E_n\cap B$ for any ball $B\subset \Omega$. Then one arrives at the full measure of $\limsup E_n$ in the light of Knopp's lemma.
\begin{lemma}[Knopp \cite{Kn}, see also Lemma 3.1.13 in \cite{DaKr}]\label{L:Knopp}
Let $I\subset \mathbb{R}$ be a bounded interval. If $B\subset I$ is a Lebesgue measurable set and $\mathcal{C}$ is a class of subintervals of $I$ satisfying
\begin{enumerate}
  \item every open subinterval of $I$ is at most a countable union of disjoint elements from $\mathcal{C}$,
  \item for any $A\in \mathcal{C}$, $\mathcal{L}(A\cap B)\geq \rho\mathcal{L}(A)$, where $\rho>0$ is a constant independent of $A$,
\end{enumerate}
then $\mathcal{L}(B)=\mathcal{L}(I)=|I|$.
\end{lemma}

\textbf{So our strategy is as follows}: For every $N\in \mathbb{N}$ and $u\in \mathbb{U}_{a_{N}}$, with the help of the Chung-Erd\"{o}s inequality, we prove that for any cylinder $I(\tau)$ contained in $I(u)$, one has
$$\mathcal{L}\Big(E_{x}(\{x_{n}\},\varphi)\cap I(\tau)\Big)\geq \rho_{u}|I(\tau)|,$$
where $\rho_{u}>0$ is a constant depending only on $x$ and $u$. Then, Knopp's Lemma enables us to conclude that $$\mathcal{L}\Big(E_{x}(\{x_{n}\},\varphi)\cap I(u)\Big)=|I(u)|.$$ Finally, by $(\ref{E:(1,+)})$, it follows that the set $E_{x}(\{x_{n}\},\varphi)$ is of full Lebesgue measure in $(1,+\infty)$.

Fix $N\in \mathbb{N}$ and $u\in \mathbb{U}_{a_{N}}$. Let $m\geq a_{N}$ and $\tau\in \Omega_{m}(x)$ with $\tau|_{a_{N}}=u$. For any $n\geq m$, set
\begin{equation*}
    \mathbb{A}_{n}(\tau)=\{w\in \Lambda_{n}(x)\colon w|_{m}=\tau\},
\end{equation*}
i.e., the collection of all $w\in \Omega_{n}(x)$ such that $I(w)$ is a full cylinder of order $n$ in the parameter space $\{\beta\in\mathbb{R}\colon \beta>1\}$ contained in $I(\tau)$.

By the definition of $\mathbb{A}_{n}(\cdot)$ and the fact that $u$ is a prefix of $\tau$, it is clear that $\mathbb{A}_{n}(\tau)\subset \mathbb{A}_{n}(u)$.
The following result says that the measure of $I(w;\varphi)$ can be well controlled when $I(w)$ is full.
\begin{proposition}\label{P:Iwp/Iw}
  For any $n\geq a_{N}=|u|$ and $w\in \mathbb{A}_{n}(u)$, we have
\begin{equation*}
    |I(w;\varphi)|\geq  \frac{1}{2}\varphi(n) |I(w)|.
\end{equation*}
\end{proposition}
\begin{proof}
By Lemma \ref{L:length}, we know that the set
\begin{equation*}
    I(w;\varphi)=\{\beta\in I(w)\colon |T^{n}_{\beta}x-x_{n}|<\varphi(n)\}
\end{equation*}
is a subinterval of $I(w)$. Recall (\ref{f2}) and (\ref{E:f'(beta)}) that on the interval $I(w)$,
the function \begin{equation*}
  f_{w}(\beta)=
  \beta^{n}\left(x-\sum_{i=1}^{n}\frac{w_{i}}{\beta^{i}}\right)
\end{equation*} is equal to $T^{n}_{\beta}x$, continuous and strictly
increasing with $f_{w}'(\beta)\geq x\beta^{n-1}>0$.

Since
$I(w)$ is a full cylinder,
one has $f_{w}(\underline{\beta}(w))=0$ and $f_{w}(\overline{\beta}(w))=1$.
Thus it follows \begin{align*}
f_w\Big(I(w;\varphi)\Big)&=f_w\Big(I(w)\Big)\cap \Big(x_n-\varphi(n), x_n+\varphi(n)\Big)\\&=[0,1)\cap \Big(x_n-\varphi(n), x_n+\varphi(n)\Big),
\end{align*} which is an interval of length no smaller than $\varphi(n)$. Then using the mean value theorem on the intervals $I(w;\varphi)$ and $I(w)$ respectively, there exist some $\beta_{\ast}\in I(w;\varphi)$ and $\beta_{\ast\ast}\in I(w)$ such that
\begin{equation*}\label{E:Iwp/Iw}
    \frac{|I(w;\varphi)|}{|I(w)|}
\geq \frac{f'_{w}(\beta_{\ast\ast})\cdot\varphi(n)}{f'_{w}(\beta_{\ast})}.
\end{equation*}
Therefore, we only need to show $f'_{w}(\beta_{\ast\ast})/f'_{w}(\beta_{\ast})\geq 1/2$.

Note that on the interval $I(w)$, we have
\begin{align*}
    0\leq f''_{w}(\beta)&=\beta^{n-2}\left(n(n-1)x-\sum_{i=1}^{n-2}\frac{(n-i)(n-i-1)w_i}{\beta^i}\right)\leq n^{2}x\beta^{n-2}.
\end{align*}
Together with $f'_{w}(\beta_{\ast})\geq x\beta_{\ast}^{n-1}>0$, we have
\begin{equation*}
    \frac{|f'_{w}(\beta_{\ast\ast})-f'_{w}(\beta_{\ast})|}{|f'_{w}(\beta_{\ast})|}\leq \frac{n^{2}x(\overline{\beta}(w))^{n-2}|\beta_{\ast\ast}-\beta_{\ast}|}{f'_{w}(\beta_{\ast})}\leq \frac{n^{2}x(\overline{\beta}(w))^{n-2}|I(w)|}{x(\underline{\beta}(w))^{n-1}}.
\end{equation*}
Then, by Lemma \ref{L:I(w)} on the length of $I(w)$, it follows that
\begin{equation*}
    \frac{|f'_{w}(\beta_{\ast\ast})-f'_{w}(\beta_{\ast})|}{|f'_{w}(\beta_{\ast})|}\leq\frac{n^{2}(\overline{\beta}(w))^{n-2}}{(\underline{\beta}(w))^{n-1}}\cdot x^{-1}(\overline{\beta}(w))^{1-n}=\frac{n^{2}}{x(\underline{\beta}(w))^{n-1}\overline{\beta}(w)}\leq \frac{n^{2}}{x\beta_{N}^{n}}\leq \frac{1}{2},
\end{equation*} where, for the last two inequalities, (\ref{E:N0}) and (\ref{E:lwlurwru}) are used.
Therefore,
\begin{equation*}
    \frac{f'_{w}(\beta_{\ast\ast})}{f'_{w}(\beta_{\ast})}=1+\frac{f'_{w}(\beta_{\ast\ast})-f'_{w}(\beta_{\ast})}{f'_{w}(\beta_{\ast})}\geq 1-\frac{|f'_{w}(\beta_{\ast\ast})-f'_{w}(\beta_{\ast})|}{|f'_{w}(\beta_{\ast})|}\geq \frac{1}{2}.
\end{equation*}
\end{proof}

For any $m\geq a_{N}$ and $\tau \in\Omega_{m}(x)$ with $\tau|_{a_{N}}=u$, we will prove the following key lemma (Lemma \ref{L:ML2}). Note that $$
E_{x}(\{x_{n}\},\varphi)\cap I(\tau)=\limsup_{l\to \infty} \textbf{E}_{l}(x)\cap I(\tau),
$$ so we have to give an effective lower bound estimation on the Lebesgue measure of $\textbf{E}_{l}(x)\cap I(\tau)$. This is possible, because
\begin{itemize}\item on one hand, Proposition \ref{P:Iwp/Iw} says that the Lebesgue measure of $I(w;\varphi)$ can be well controlled when $I(w)$ is full;

\item on the other hand, Lemma \ref{L:CcC} guarantees a sufficient large collection of full cylinders. \end{itemize}

This renders us a nice subset $F_{l}(\tau;\varphi)$ of $\textbf{E}_{l}(x)\cap I(\tau)$, which will be given in detail in Subsection $6.1$. Then, in Subsection $6.2$, we shall estimate the Lebesgue measure of the set $F_{l}(\tau;\varphi)$. The proofs of Lemma \ref{L:ML2} and the divergent part of Theorem \ref{T:main} are given in Subsection $6.3$. As what we will see, full cylinders play essential roles in the estimation of the lower bound of the Lebesgue measure of the set $E_{x}(\{x_{n}\},\varphi)\cap I(\tau)$.

\begin{lemma}\label{L:ML2}
  For any $m\geq a_{N}$ and $\tau \in\Omega_{m}(x)$ with $\tau|_{a_{N}}=u$, we have
  \begin{equation*}
    \mathcal{L}(E_{x}(\{x_{n}\},\varphi)\cap I(\tau))\geq \rho_{u}|I(\tau)|,
  \end{equation*}
where $\rho_{u}>0$ is a constant depending only on $x$ and $u$.
\end{lemma}

\subsection{Structure of the set $F_{l}(\tau;\varphi)$} Recall that $u\in \mathbb{U}_{a_{N}}$ is fixed and $\underline{\beta}(u)\geq \beta_{N}>1$. Let $\lambda_{u}\in(0, 1)$ be a real number such that
   \begin{equation*}
\lambda_{u}-\lambda_{u}\ln \lambda_{u}<\frac{\left(\underline{\beta}(u)-1\right)^{2}}{\left(\overline{\beta}(u)\right)^{3}}\quad\text{and}\quad\lambda_u<(\overline{\beta}(w))^{-1},
\end{equation*}
to fulfill the conditions in Proposition \ref{P:numperf}.
Let $c_{u}$ be a positive integer large enough such that $c_{u}\geq-\log_{\underline{\beta}(u)}\lambda_{u}$ and for all $n\geq c_{u}$, we have
\begin{equation}\label{E:Lu}
    \lambda_{u}\cdot \left(\underline{\beta}(u)\right)^{n}\geq 3n\geq \frac{6\underline{\beta}(u)}{\underline{\beta}(u)-1}.
\end{equation}
Then by Proposition \ref{P:numperf}, for any $\beta\in I(u)$, we have
\begin{equation}\label{E:Xi>Si}
    \#\Xi_{n}(\beta)\geq \lambda_{u}\#\Sigma_{n}(\beta)\quad\text{ for all } n\geq c_{u}.
\end{equation}

Fix $m\geq a_{N}$ and $\tau \in\Omega_{m}(x)$ with $\tau|_{a_{N}}=u$. Recall the choice of $b_{\tau}$ in (\ref{E:N1}).
Now we are intended to search for a nice subset of $\textbf{E}_{l}(x)\cap I(\tau)$ for any $l\geq 4b_{\tau}+2c_{u}$.
A potential candidate is the set $$
\widetilde{F}_l(\tau; \varphi):=\bigcup_{v\in \mathbb{A}_{l}(\tau)}I(v;\varphi),
$$ where as defined before$$\mathbb{A}_{l}(\tau)=\{w\in \Lambda_{l}(x)\colon w|_{m}=\tau\}.$$
We will give some further modification to cut off some unpleasant parts, which will facilitate the estimation of the covariance $\mathcal{L}(\widetilde{F}_n(\tau; \varphi)\cap \widetilde{F}_l(\tau; \varphi))$ later.

For any $m\leq n\leq l-c_{u}$ and $w\in\mathbb{A}_{n}(\tau)$, define
\begin{equation*}
    \mathbb{H}_{l}(w)=\{v\in  \mathbb{A}_{l}(\tau)\colon I(w;\varphi)\cap I(v;\varphi)\neq \varnothing\}.
\end{equation*}
In fact, the words in $\mathbb{H}_{l}(w)$ are nothing but those which can contribute to the Lebesgue measure of $\widetilde{F}_n(\tau; \varphi)\cap \widetilde{F}_l(\tau; \varphi)$, since $$
\widetilde{F}_n(\tau; \varphi)\cap \widetilde{F}_l(\tau; \varphi)=\bigcup_{w\in \mathbb{A}_{n}(\tau)}\left[\bigcup_{v\in  \mathbb{A}_{l}(\tau)}\left(I(w;\varphi)\cap I(v;\varphi)\right)\right].
$$

Since $I(v;\varphi)\subset I(v)$ and $I(w;\varphi)\subset I(w)$, the non-emptyness of
$I(w;\varphi)\cap I(v;\varphi)$ implies $v|_{n}=w$.
Thus $\mathbb{H}_{l}(w)\subset \mathbb{A}_{l}(w)\subset \mathbb{A}_{l}(\tau)$.
For the position relations between the sets $I(u), I(\tau), I(w)$ and $I(v)$, we have the following diagram:
\begin{center}
  \unitlength 3pt
\begin{picture}(115,37)
\linethickness{1pt}
  \put(12,33){\line(1,0){67}}
  \put(0.5,32){$I(u):$}
  \put(83,32){$|u|=a_{N}$}

  \put(25,26){\line(1,0){44}}
  \put(0.5,25){$I(\tau):$}
  \put(83,25){$|\tau|=m\geq a_{N}$}

  \put(34,15){\line(1,0){32}}
  \put(0,14){$I(w):$}
   \put(46,18){$I(w;\varphi)$}
  \put(83,14){$m\leq |w|=n\leq l-c_{u}$}

  \put(34,4){\line(1,0){32}}
  \put(0.5,3){$I(v):$}
  \put(38.5,8){$I(v';\varphi)$}
  \put(49.5,8){$I(v;\varphi)$}
  \put(59.5,8){$I(v'';\varphi)$}
  \put(83,3){$|v|=l\geq 4b_{\tau}+2c_{u}.$}

    \put(40.5,3.5){\line(0,1){1}}
    \put(49,3.5){\line(0,1){1}}
    \put(58,3.5){\line(0,1){1}}

\linethickness{3pt}
  \put(41,15){\line(1,0){22}}
  \put(35,4){\line(1,0){4}}
  \put(43,4){\line(1,0){5}}
  \put(51,4){\line(1,0){5}}
  \put(60,4){\line(1,0){5}}
\end{picture}
\end{center}
Since $I(w;\varphi)$ is an interval, for any $v\in \mathbb{H}_{l}(w)$ but not the lexicographically smallest one nor the lexicographically largest one in $\mathbb{H}_{l}(w)$, one has \begin{equation}\label{f5}
I(w;\varphi)\supset I(v)\supset I(v;\varphi).
\end{equation} So we discard the (two) elements at the boundary of $\mathbb{H}_{l}(w)$. More precisely, \begin{itemize}\item
if $\#\mathbb{H}_{l}(w)\leq 1$, define
$\mathbb{H}^{\ast}_{l}(w)=\mathbb{H}_{l}(w)$;

\item if $\#\mathbb{H}_{l}(w)\ge 2$, define $\mathbb{H}^{\ast}_{l}(w)$ to be the set consisting of the lexicographically smallest and largest elements from $\mathbb{H}_{l}(w)$. For example, if $\mathbb{H}_{l}(w)=\{v(1)\prec v(2)\prec \cdots\prec v(k)\}$ with $k\geq 2$, then $\mathbb{H}^{\ast}_{l}(w)=\{v(1), v(k)\}$.
\end{itemize}
Finally, let
\begin{equation}\label{E:Bltau}
     \mathbb{B}_{l}(\tau)=\mathbb{A}_{l}(\tau)\backslash \bigcup_{m\leq n\leq l-c_{u}}\bigcup_{w\in \mathbb{A}_{n}(\tau)}\mathbb{H}^{\ast}_{l}(w).
\end{equation}
The choice of $c_u$ makes that $\#\mathbb{B}_{l}(\tau)$ is almost the same as $\# \mathbb{A}_{l}(\tau)$ (see (\ref{E:BlAl}) below). Then the desired subset of $\textbf{E}_{l}(x)\cap I(\tau)$ is defined as
\begin{equation}\label{E:Fl}
    F_{l}(\tau;\varphi):=\bigcup_{v\in \mathbb{B}_{l}(\tau)}I(v;\varphi).
\end{equation}


\subsection{Lebesgue measure of the set $F_{l}(\tau;\varphi)$}
In this subsection, we will give the proofs of the following two propositions.
\begin{proposition}\label{P:geq}
Let $m\geq a_{N}$ and $\tau \in\Omega_{m}(x)$ with $\tau|_{a_{N}}=u$. For any $l\geq 4b_{\tau}+2c_{u}$, we have
\begin{equation*}
    \mathcal{L}(F_{l}(\tau;\varphi))\geq D_{u}\varphi(l)|I(\tau)|,\quad\text{ where } D_{u}= \frac{x\lambda_{u}}{243}\cdot\frac{(\underline{\beta}(u)-1)^{2}}{(\overline{\beta}(u))^{3}}>0.
\end{equation*}
\end{proposition}

\begin{proposition}\label{P:leq}
 Let $m\geq a_{N}$ and $\tau \in\Omega_{m}(x)$ with $\tau|_{a_{N}}=u$. For any $n\geq 4b_{\tau}+2c_{u}$ and $l\geq n+c_{u}$, we have
\begin{equation*}
    \mathcal{L}(F_{n}(\tau;\varphi)\cap F_{l}(\tau;\varphi))\leq \frac{K_{u}}{|I(\tau)|}\mathcal{L}(F_{n}(\tau;\varphi))\mathcal{L}(F_{l}(\tau;\varphi)),
\end{equation*}
where
\begin{equation*}
  K_{u}=\frac{1458}{x^{2}\lambda_{u}}\cdot\frac{(\overline{\beta}(u))^{5}}{\left(\underline{\beta}(u)-1\right)^{4}}>0.
\end{equation*}
\end{proposition}

\begin{proof}[Proof of Proposition \ref{P:geq}]
Since $\mathbb{B}_{l}(\tau)\subset \mathbb{A}_{l}(\tau)\subset \mathbb{A}_{l}(u)$ (see (\ref{E:Bltau})), by Proposition \ref{P:Iwp/Iw}, we obtain
\begin{equation}\label{E:LBlu}
    \mathcal{L}(F_{l}(\tau;\varphi))=\sum_{v\in \mathbb{B}_{l}(\tau)}|I(v;\varphi)|\geq \frac{1}{2}\varphi(l)\sum_{v\in \mathbb{B}_{l}(\tau)}|I(v)|.
\end{equation}

Choose $q=\lfloor l/4\rfloor$, so $q\geq b_{\tau}\geq |\tau|=m$. It is trivial that
\begin{equation*}\label{E:Itau=IoIo}
    I(\tau)=\bigcup_{\varpi\in \Omega_{q}(x), \varpi|_{m}=\tau}I(\varpi)=\bigcup_{\varpi\in \mathbb{P}_{q}(\tau)}I(\varpi) \cup \bigcup_{\varpi\in \widehat{\mathbb{P}}_{q}(\tau)}I(\varpi),
\end{equation*}
where all of the unions are disjoint. Then by Proposition \ref{P:1/n2}, it follows that
\begin{equation*}
    |I(\tau)|=\sum_{\varpi\in \mathbb{P}_{q}(\tau)}|I(\varpi)|+\sum_{\varpi\in \widehat{\mathbb{P}}_{q}(\tau)}|I(\varpi)|\leq \sum_{\varpi\in \mathbb{P}_{q}(\tau)}|I(\varpi)|+\frac{1}{q^{2}}|I(\tau)|.
\end{equation*}
Since $q\geq m\geq a_{N}\geq 2$, it follows that
\begin{equation}\label{E:Itau<}
    \sum_{\varpi\in \mathbb{P}_{q}(\tau)}|I(\varpi)|\geq \left(1-\frac{1}{q^{2}}\right)|I(\tau)|\geq \frac{2}{3}|I(\tau)|.
\end{equation}

On the other hand, note that
\begin{align*}
    \bigcup_{v\in \mathbb{B}_{l}(\tau)}I(v)&=I(\tau)\cap \bigcup_{v\in \mathbb{B}_{l}(\tau)}I(v)\supset \bigcup_{\varpi\in \mathbb{P}_{q}(\tau)}I(\varpi)\cap
    \bigcup_{v\in \mathbb{B}_{l}(\tau)}I(v)\\
    &\supset\bigcup_{\varpi\in \mathbb{P}_{q}(\tau)}\bigcup_{v\in \mathbb{B}_{l}(\tau)\cap \mathbb{A}_{l}(\varpi)}(I(\varpi)\cap I(v))\\
    &=\bigcup_{\varpi\in \mathbb{P}_{q}(\tau)}\bigcup_{v\in \mathbb{B}_{l}(\tau)\cap \mathbb{A}_{l}(\varpi)}I(v),
\end{align*}
where all of the unions are disjoint. Thus,
\begin{equation}\label{E:sumIv}
    \sum_{v\in \mathbb{B}_{l}(\tau)}|I(v)|\geq \sum_{\varpi\in \mathbb{P}_{q}(\tau)}\sum_{v\in \mathbb{B}_{l}(\tau)\cap \mathbb{A}_{l}(\varpi)}|I(v)|.
\end{equation}
Next, we hope to show that
\begin{equation*}
 \sum_{v\in \mathbb{B}_{l}(\tau)\cap \mathbb{A}_{l}(\varpi)}|I(v)|\geq 3D_{u}|I(\varpi)|\quad\text{ for all } \varpi\in \mathbb{P}_{q}(\tau),
\end{equation*}
which together with $(\ref{E:LBlu})$, $(\ref{E:sumIv})$ and $(\ref{E:Itau<})$ one by one implies
\begin{align*}
  \mathcal{L}(F_{l}(\tau;\varphi))&\geq \frac{1}{2}\varphi(l)\sum_{v\in \mathbb{B}_{l}(\tau)}|I(v)|\ge \frac{1}{2}\varphi(l)\sum_{\varpi\in \mathbb{P}_{q}(\tau)}\sum_{v\in \mathbb{B}_{l}(\tau)\cap \mathbb{A}_{l}(\varpi)}|I(v)|\\
  &\ge \frac{3}{2}\cdot D_{u}\varphi(l)\sum_{\varpi\in \mathbb{P}_{q}(\tau)}|I(\varpi)|\ge
  D_{u}\varphi(l)|I(\tau)|.
\end{align*}

Fix an arbitrary $\varpi\in \mathbb{P}_{q}(\tau)$. The proof is divided into 4 steps.

\textbf{Step 1.}  We give a lower bound of $\#\mathbb{A}_{l}(\varpi)$. Let $$k_{0}=\min\{k\geq 0\colon \varpi 0^{k}1\in \Omega_{q+k+1}(x)\}.$$ By the definition of $\mathbb{P}_{q}(\tau)$, one has $0\leq k_{0}\leq q-1$. In the same way as in the proof of Lemma \ref{L:ML1}, we have
$$I(\varpi)=I(\varpi0^{k_{0}})\  {\text{and}}\ \varpi0^{k_{0}+1}\in \Lambda_{q+k_{0}+1}(x).$$
So by Lemma \ref{L:CcC}, for any $\varsigma\in \Xi_{l-q-k_{0}-1}\left(\underline{\beta}(\varpi 0^{k_{0}+1})\right)$,
 one has $\varpi 0^{k_{0}+1}\varsigma\in \Lambda_{l}(x)$. Thus
\begin{equation*}\label{E:AlXi}
    \#\mathbb{A}_{l}(\varpi)\geq \#\Xi_{l-q-k_{0}-1}\left(\underline{\beta}(\varpi 0^{k_{0}+1})\right).
\end{equation*}
At the same time
\begin{equation}\label{E:l-q-k}
    l-q-k_{0}-1\geq l-2q\geq l/2\geq c_{u},
\end{equation}
then by $(\ref{E:Xi>Si})$ it follows $$\#\Xi_{l-q-k_{0}-1}\left(\underline{\beta}(\varpi 0^{k_{0}+1})\right)\geq \lambda_{u}\#\Sigma_{l-q-k_{0}-1}\left(\underline{\beta}(\varpi 0^{k_{0}+1})\right).$$ Therefore, by
Lemma \ref{L:number}, it follows that
\begin{align}\label{E:numA}
    \#\mathbb{A}_{l}(\varpi)&\geq \lambda_{u}\#\Sigma_{l-q-k_{0}-1}\left(\underline{\beta}(\varpi 0^{k_{0}+1})\right)\\
    &\geq \lambda_{u}\left(\underline{\beta}(\varpi 0^{k_{0}+1})\right)^{l-q-k_{0}-1}\geq \lambda_{u}(\underline{\beta}(\varpi))^{l-q-k_{0}-1}.\nonumber
\end{align}

\textbf{Step 2.} We compare $\#\mathbb{A}_{n}(\varpi)$ with $\#\mathbb{A}_{l}(\varpi)$ for every $q\leq n\leq l-c_{u}$. Fix $q\leq n\leq l-c_{u}$. Note that for each $w\in \mathbb{A}_{n}(\varpi)$, $I(w)$ is a full cylinder, i.e., $w\in \Lambda_{n}(x)$. So by Lemma \ref{L:CcC}, $w\varsigma\in \mathbb{A}_{l}(\varpi)$ for all
$\varsigma\in\Xi_{l-n}(\underline{\beta}(w))$. This implies each $w\in \mathbb{A}_{n}(\varpi)$
can contribute at least $\#\Xi_{l-n}(\underline{\beta}(w))$ elements to $\mathbb{A}_{l}(\varpi)$.
Moreover, since $l-n\geq c_{u}$, by $(\ref{E:Xi>Si})$ and Lemma \ref{L:number}, it follows that
$$\#\Xi_{l-n}(\underline{\beta}(w))\geq \lambda_{u}\#\Sigma_{l-n}(\underline{\beta}(w))\geq \lambda_{u}(\underline{\beta}(w))^{l-n}\geq \lambda_{u}(\underline{\beta}(u))^{l-n}.$$
Therefore, we have
\begin{equation}\label{E:A<A}
\lambda_{u}(\underline{\beta}(u))^{l-n}\cdot\#\mathbb{A}_{n}(\varpi)\leq \#\mathbb{A}_{l}(\varpi).
\end{equation}

\textbf{Step 3.} We give a lower bound of $\#(\mathbb{B}_{l}(\tau)\cap \mathbb{A}_{l}(\varpi))$.
Observe that $\mathbb{A}_{l}(\varpi)\subset \mathbb{A}_{l}(\tau)$, since $\varpi|_m=\tau$.
Recalling $(\ref{E:Bltau})$, it follows
\begin{align*}
    \mathbb{B}_{l}(\tau)\cap \mathbb{A}_{l}(\varpi)&=\left(\mathbb{A}_{l}(\tau)\backslash \bigcup_{m\leq n\leq l-c_{u}}\bigcup_{w\in \mathbb{A}_{n}(\tau)}\mathbb{H}^{\ast}_{l}(w)\right)\cap\mathbb{A}_{l}(\varpi)\\
&=\mathbb{A}_{l}(\varpi)\backslash \bigcup_{m\leq n\leq l-c_{u}}\bigcup_{w\in \mathbb{A}_{n}(\tau)}\mathbb{H}^{\ast}_{l}(w)\\
&=\mathbb{A}_{l}(\varpi)\cap \left(\bigcup_{m\leq n<q}\bigcup_{w\in \mathbb{A}_{n}(\tau)}\mathbb{H}^{\ast}_{l}(w)\right)^{c}\cap \left(\bigcup_{q\leq n\leq l-c_{u}}\bigcup_{w\in \mathbb{A}_{n}(\tau)}\mathbb{H}^{\ast}_{l}(w)\right)^{c},
\end{align*}
where $(\cdot)^{c}$ denotes the complement.

Note that both $\varpi$ and $w$ are the common prefixes of the words in $\mathbb{A}_{l}(\varpi)\cap\mathbb{H}^{\ast}_{l}(w)$.
So if  $\mathbb{A}_{l}(\varpi)\cap\mathbb{H}^{\ast}_{l}(w)\neq \varnothing$, one must have $w=\varpi|_{n}$ if $n<q$, and $\varpi=w|_{q}$ if $n\geq q$. Thus
\begin{equation*}
    \mathbb{B}_{l}(\tau)\cap \mathbb{A}_{l}(\varpi)=\mathbb{A}_{l}(\varpi)\cap \left(\bigcup_{m\leq n<q}\bigcup_{\substack{w\in \mathbb{A}_{n}(\tau),\\w=\varpi|_{n}}}\mathbb{H}^{\ast}_{l}(w)\right)^{c}\cap \left(\bigcup_{q\leq n\leq l-c_{u}}\bigcup_{\substack{w\in \mathbb{A}_{n}(\tau),\\w|_{q}=\varpi}}\mathbb{H}^{\ast}_{l}(w)\right)^{c}.
\end{equation*}
Note that for any $n\geq q$ and $w\in \mathbb{A}_{n}(\tau)$ with $w|_{q}=\varpi$, we also have $w\in \mathbb{A}_{n}(\varpi)$. Since $\#\mathbb{H}^{\ast}_{l}(w)\leq 2$, one has
\begin{equation*}
    \#(\mathbb{B}_{l}(\tau)\cap \mathbb{A}_{l}(\varpi))\geq \#\mathbb{A}_{l}(\varpi)-2(q-m)-2\sum_{q\leq n\leq l-c_{u}}\#\mathbb{A}_{n}(\varpi).
\end{equation*}
Thus, by $(\ref{E:A<A})$ and at last by $(\ref{E:Lu})$, it follows that
\begin{align}\label{E:BlAl}
    \#(\mathbb{B}_{l}(\tau)\cap \mathbb{A}_{l}(\varpi))&\geq \#\mathbb{A}_{l}(\varpi)-2q-2\lambda^{-1}_{u}\#\mathbb{A}_{l}(\varpi)\sum_{q\leq n\leq l-c_{u}}(\underline{\beta}(u))^{n-l}\\\nonumber
    &\geq \#\mathbb{A}_{l}(\varpi)-\frac{l}{2}-\frac{2\#\mathbb{A}_{l}(\varpi)}{\lambda_{u}(\underline{\beta}(u))^{c_{u}-1}(\underline{\beta}(u)-1)}\\\nonumber
    &\geq \frac{2}{3}\#\mathbb{A}_{l}(\varpi)-\frac{l}{2}.
\end{align}
On the other hand, by $(\ref{E:Lu})$, $(\ref{E:l-q-k})$ and $(\ref{E:numA})$, we have
\begin{equation*}
    \#\mathbb{A}_{l}(\varpi)\geq\lambda_{u}(\underline{\beta}(\varpi))^{l-q-k_{0}-1}\geq \lambda_{u}(\underline{\beta}(\varpi))^{l/2}\geq \lambda_{u}(\underline{\beta}(u))^{l/2}\geq 3l/2.
\end{equation*}
Thus it follows that
\begin{equation}\label{E:BlAl/3}
    \#(\mathbb{B}_{l}(\tau)\cap \mathbb{A}_{l}(\varpi))\geq \frac{1}{3}\#\mathbb{A}_{l}(\varpi)\geq \frac{1}{3}\lambda_{u}(\underline{\beta}(\varpi))^{l-q-k_{0}-1}.
\end{equation}

\textbf{Step 4.} We show
\begin{equation*}
   \sum_{v\in \mathbb{B}_{l}(\tau)\cap \mathbb{A}_{l}(\varpi)}|I(v)|\geq 3D_{u}|I(\varpi)|.
\end{equation*}
By Proposition \ref{P:Cfulen} and $(\ref{E:BlAl/3})$, we have
\begin{align*}
    \sum_{v\in \mathbb{B}_{l}(\tau)\cap \mathbb{A}_{l}(\varpi)}|I(v)|&\geq \sum_{v\in \mathbb{B}_{l}(\tau)\cap \mathbb{A}_{l}(\varpi)} (\underline{\beta}(v)-1)^{2}(\overline{\beta}(v))^{-1-l}\\
&\geq \sum_{v\in \mathbb{B}_{l}(\tau)\cap \mathbb{A}_{l}(\varpi)} (\underline{\beta}(\varpi)-1)^{2}(\overline{\beta}(\varpi))^{-1-l}\\
&\geq \frac{1}{3}\lambda_{u}(\underline{\beta}(\varpi))^{l-q-k_{0}-1}(\underline{\beta}(\varpi)-1)^{2}(\overline{\beta}(\varpi))^{-1-l}\\
&=\frac{\lambda_{u}}{3}\cdot\frac{(\underline{\beta}(\varpi)-1)^{2}}{\underline{\beta}(\varpi) (\overline{\beta}(\varpi))^{2}}\cdot\left(\frac{\underline{\beta}(\varpi)}{\overline{\beta}(\varpi)}\right)^{l-q-k_{0}}\cdot(\overline{\beta}(\varpi))^{1-q-k_{0}}\\
&\geq \frac{\lambda_{u}}{3} \cdot \frac{(\underline{\beta}(u)-1)^{2}}{(\overline{\beta}(u))^{3}}\cdot
\left(\frac{\underline{\beta}(\varpi)}{\overline{\beta}(\varpi)}\right)^{3q+3}\cdot(\overline{\beta}(\varpi))^{1-q-k_{0}}.
\end{align*}
Recall that $I(\varpi)=I(\varpi 0^{k_{0}})$, which implies $\overline{\beta}(\varpi)=\overline{\beta}(\varpi 0^{k_{0}})$. Therefore, by Lemma \ref{L:I(w)} on the length of a cylinder and Proposition \ref{P:ov/un<3}, it follows that
\begin{align*}
    \sum_{v\in \mathbb{B}_{l}(\tau)\cap \mathbb{A}_{l}(\varpi)}|I(v)|&\geq \frac{x\lambda_{u}}{81}\cdot \frac{(\underline{\beta}(u)-1)^{2}}{(\overline{\beta}(u))^{3}}|I(\varpi 0^{k_{0}})|\\
&=\frac{x\lambda_{u}}{81}\cdot \frac{(\underline{\beta}(u)-1)^{2}}{(\overline{\beta}(u))^{3}}|I(\varpi)|=3D_{u}|I(\varpi)|.
\end{align*}
\end{proof}

\begin{proof}[Proof of Proposition \ref{P:leq}]
By $(\ref{E:Fl})$ on the definition of $F_n(\tau;\varphi)$, we have
\begin{equation*}
    F_{n}(\tau;\varphi)\cap F_{l}(\tau;\varphi)=\bigcup_{w\in \mathbb{B}_{n}(\tau)}(I(w;\varphi)\cap F_{l}(\tau;\varphi)),
\end{equation*}
where the union is disjoint. So
\begin{equation}\label{E:FF}
    \mathcal{L}(F_{n}(\tau;\varphi)\cap F_{l}(\tau;\varphi))=\sum_{w\in \mathbb{B}_{n}(\tau)}\mathcal{L}(I(w;\varphi)\cap F_{l}(\tau;\varphi)).
\end{equation}
We will prove that for any $w\in \mathbb{B}_{n}(\tau)$, $$\mathcal{L}(I(w;\varphi)\cap F_{l}(\tau;\varphi))\leq D_{u}K_{u}\varphi(l)|I(w;\varphi)|$$ which together with $(\ref{E:FF})$ and Proposition \ref{P:geq} implies
\begin{align*}
     \mathcal{L}(F_{n}(\tau;\varphi)\cap F_{l}(\tau;\varphi))&
     \le D_{u}K_{u}\varphi(l)\sum_{w\in \mathbb{B}_{n}(\tau)}|I(w;\varphi)|=D_{u}K_{u}\varphi(l)\mathcal{L}(F_{n}(\tau;\varphi))\\
     &\leq \frac{K_{u}}{|I(\tau)|}\mathcal{L}(F_{n}(\tau;\varphi))\mathcal{L}(F_{l}(\tau;\varphi)).
\end{align*}

Fix an arbitrary $w\in \mathbb{B}_{n}(\tau)$. Note that
\begin{equation*}
    I(w;\varphi)\cap F_{l}(\tau;\varphi)
=\bigcup_{v\in \mathbb{B}_{l}(\tau)}(I(w;\varphi)\cap I(v;\varphi))=\bigcup_{v\in \mathbb{D}_{l}(w)}(I(w;\varphi)\cap I(v;\varphi)),
\end{equation*}
where we set
\begin{equation*}
    \mathbb{D}_{l}(w)=\{v\in \mathbb{B}_{l}(\tau)\colon I(w;\varphi)\cap I(v;\varphi)\neq \varnothing\}.
\end{equation*}
Then,
\begin{equation}\label{E:IwF}
    \mathcal{L}(I(w;\varphi)\cap F_{l}(\tau;\varphi))
=\sum_{v\in\mathbb{D}_{l}(w)}|I(w;\varphi)\cap I(v;\varphi)|.
\end{equation}

By the definition of $\mathbb{B}_{l}(\tau)$ (see $(\ref{E:Bltau})$), it is easy to see that $\mathbb{D}_{l}(w)\subset \mathbb{H}_{l}(w)\backslash \mathbb{H}^{\ast}_{l}(w)$.
So by the design of $\mathbb{H}^{\ast}_{l}(w)$ and (\ref{f5}), one has
\begin{equation*}\label{E:|IwIv|}
    |I(w;\varphi)|\geq \sum_{v\in \mathbb{D}_{l}(w)}|I(v)|.
\end{equation*}

On the other hand, for $v\in \mathbb{D}_{l}(w)$, recall the lengths of $I(v;\varphi)$ (Lemma \ref{L:length}) and $I(v)$ (Proposition \ref{P:Cfulen}). Then by Proposition \ref{P:ov/un<3}, we have
\begin{align}\label{E:IwIv}
    |I(w;\varphi)\cap I(v;\varphi)|&= |I(v;\varphi)|\leq 2x^{-1}\varphi(l)(\underline{\beta}(v))^{1-l}\\\nonumber
    &=\frac{2\varphi(l)(\underline{\beta}(v))^{2}}{x\left(\underline{\beta}(v)-1\right)^{2}}\cdot
    \left(\frac{\overline{\beta}(v)}{\underline{\beta}(v)}\right)^{l+1}\cdot\left(\underline{\beta}(v)-1\right)^{2}(\overline{\beta}(v))^{-1-l}\\\nonumber
    &\leq \frac{6\varphi(l)(\underline{\beta}(v))^{2}}{x\left(\underline{\beta}(v)-1\right)^{2}}|I(v)|\leq \frac{6\varphi(l)(\overline{\beta}(u))^{2}}{x\left(\underline{\beta}(u)-1\right)^{2}}|I(v)|\\&=D_{u}K_{u}\varphi(l)|I(v)|.\nonumber
\end{align}
Hence, by $(\ref{E:IwF})$--$(\ref{E:IwIv})$, we have
\begin{equation*}
    \frac{\mathcal{L}(I(w;\varphi)\cap F_{l}(\tau;\varphi))}{|I(w;\varphi)|}\leq \frac{\sum_{v\in \mathbb{D}_{l}(w)}|I(w;\varphi)\cap I(v;\varphi)|}{\sum_{v\in \mathbb{D}_{l}(w)}|I(v)|}\leq D_{u}K_{u}\varphi(l).
\end{equation*}
Therefore, $\mathcal{L}(I(w;\varphi)\cap F_{l}(\tau;\varphi))\leq D_{u}K_{u}\varphi(l)|I(w;\varphi)|$.
\end{proof}

\subsection{Proofs of the main results}
Recall that $N\in\mathbb{N}$ and $u\in \mathbb{U}_{a_{N}}$ are fixed.
\begin{proof}[Proof of Lemma \ref{L:ML2}]
Fix $m\geq a_{N}$ and $\tau \in\Omega_{m}(x)$ with $\tau|_{a_{N}}=u$.
Having Proposition \ref{P:geq} and Proposition \ref{P:leq} in hand, we can then apply the Chung-Erd\"{o}s inequality to the limsup set $$
\limsup_{l\to \infty} F_{l}(\tau;\varphi).
$$ More precisely, Proposition \ref{P:geq} ensures that $$
\sum_{l=4b_{\tau}+2c_{u}}^{\infty}\mathcal{L}(F_{l}(\tau;\varphi))=\infty
$$ so the condition of Lemma \ref{L:PZI} is met. One the other hand, by Proposition \ref{P:leq}\begin{align*}
&\sum_{4b_{\tau}+2c_{u}\le n<l\le N}\mathcal{L}(F_{l}(\tau;\varphi)\cap F_{n}(\tau;\varphi))\\
=&\left(\sum_{n=4b_{\tau}+2c_{u}}^N\sum_{l= n+c_u}^N+\sum_{n=4b_{\tau}+2c_{u}}^N\sum_{l= n+1}^{n+c_u-1}\right)\mathcal{L}(F_{l}(\tau;\varphi)\cap F_{n}(\tau;\varphi))\\
\le& \frac{1}{2}\cdot \frac{K_{u}}{|I(u)|}\left(\sum_{n=4b_{\tau}+2c_{u}}^N\mathcal{L}(F_{n}(\tau;\varphi))\right)^2+
c_u\sum_{n=4b_{\tau}+2c_{u}}^N\mathcal{L}(F_{n}(\tau;\varphi)).
\end{align*}
Thus by Chung-Erd\"{o}s inequality one has
\begin{equation*}
    \mathcal{L}(E_{x}(\{x_{n}\},\varphi)\cap I(\tau))\geq \mathcal{L}(\limsup_{l\to\infty}F_l(\tau; \varphi))\ge \rho_{u}|I(\tau)|,
\end{equation*}
where the constant $\rho_{u}=1/K_{u}>0$ only depends on $x$ and $u$.
\end{proof}

\begin{proof}[Proof of Theorem \ref{T:main}: the divergent part]
 Let $N\in\mathbb{N}$ and $u\in \mathbb{U}_{a_{N}}$.
Let $\mathcal{C}$ be the collection of all cylinders in the parameter space $\{\beta\in \mathbb{R}\colon\beta>1\}$ contained in $I(u)$.
Then, in light of Lemma \ref{L:ML2}, the collection $\mathcal{C}$ satisfies the conditions in Knopp's lemma.
Thus, we obtain $$\mathcal{L}(E_{x}(\{x_{n}\}, \varphi)\cap I(u))=|I(u)|.$$ Therefore,
the set $E_{x}(\{x_{n}\}, \varphi)$ is of full Lebesgue measure in $(1,+\infty)$.
\end{proof}

\section{Proofs of Corollary \ref{T:ln} and Theorem \ref{T:rec}}
In this section, we will give the proofs of Corollary \ref{T:ln} and Theorem \ref{T:rec}.

\subsection{Diophantine analysis in parameter space}
Recall
\begin{equation*}
    \mathcal{E}_{x}(\{x_{n}\},\{l_{n}\})=\{\beta>1\colon |T^{n}_{\beta}x-x_{n}|<\beta^{-l_{n}} \text{ for infinitely many } n\in\mathbb{N}\},
\end{equation*}
and\begin{equation*}
    \beta^{\star}=\inf\left\{\beta>1\colon \sum\beta^{-l_{n}}<+\infty\right\}=\sup\left\{\beta>1\colon \sum\beta^{-l_{n}}=+\infty\right\},
\end{equation*}
where $\inf\varnothing=+\infty$ and $\sup\varnothing=1$ for the empty set $\varnothing$.

\begin{proof}[Proof of Corollary \ref{T:ln}]
We shall prove that \begin{itemize}\item
for any subinterval $[s,t]\subset (1,\beta^{\star})$,  $$\mathcal{L}\left(\mathcal{E}_{x}(\{x_{n}\},\{l_{n}\})\cap [s,t]\right)=t-s;$$

  \item for any subinterval $[s,t]\subset (\beta^{\star},+\infty)$,
  $$\mathcal{L}\left(\mathcal{E}_{x}(\{x_{n}\},\{l_{n}\})\cap [s,t]\right)=0.$$
  \end{itemize}
      This enables us to conclude that $$\mathcal{L}(\mathcal{E}_{x}(\{x_{n}\},\{l_{n}\}))=\beta^{\star}-1$$ no matter what $\beta^{\star}\in [1,\infty]$ is.

  Note  that for any subinterval $[s,t]\subset (1,\beta^{\star})$,
  \begin{align*}
    \mathcal{E}_{x}(\{x_{n}\},\{l_{n}\})\cap [s,t]&=\{\beta\in[s,t]\colon |T^{n}_{\beta}x-x_{n}|<\beta^{-l_{n}} \text{ for infinitely many } n\in\mathbb{N}\}\\
    &\supset \{\beta\in[s,t]\colon |T^{n}_{\beta}x-x_{n}|<t^{-l_{n}} \text{ for infinitely many } n\in\mathbb{N}\}\\
    &=E_{x}(\{x_{n}\},\varphi)\cap [s,t],
  \end{align*}
where $\varphi(n)=t^{-l_{n}}$ for all $n\in \mathbb{N}$. Since $1<t<\beta^{\star}$, we know that $$\sum_{n=1}^{\infty}\varphi(n)=\sum_{n=1}^{\infty} t^{-l_{n}}=+\infty.$$ Then by Theorem \ref{T:main}, it follows that
\begin{equation*}
    \mathcal{L}(\mathcal{E}_{x}(\{x_{n}\},\{l_{n}\})\cap [s,t])=\mathcal{L}(E_{x}(\{x_{n}\},\varphi)\cap [s,t])=t-s.
\end{equation*}

Similarly, for any subinterval $[s,t]\subset (\beta^{\star},+\infty)$, we have
  \begin{align*}
    \mathcal{E}_{x}(\{x_{n}\},\{l_{n}\})\cap [s,t]&
    \subset \{\beta\in[s,t]\colon |T^{n}_{\beta}x-x_{n}|<s^{-l_{n}} \text{ for infinitely many } n\in\mathbb{N}\}\\
    &\subset E_{x}(\{x_{n}\},\widetilde{\varphi}),
  \end{align*}
where $\widetilde{\varphi}(n)=s^{-l_{n}}$ for all $n\in \mathbb{N}$. Since $s>\beta^{\star}\geq 1$, we know that $\sum\widetilde{\varphi}(n)<+\infty$. Then Theorem \ref{T:main} gives that
\begin{equation*}
    \mathcal{L}(\mathcal{E}_{x}(\{x_{n}\},\{l_{n}\})\cap [s,t])=\mathcal{L}(E_{x}(\{x_{n}\},\widetilde{\varphi}))=0.
\end{equation*}
\end{proof}

\subsection{Quantitative recurrence in beta-expansion}
Let $\varphi\colon \mathbb{N}\rightarrow (0,1]$ be a positive function. For any $\beta>1$, let
\begin{equation*}
    \mathfrak{R}_{\beta}(\varphi)=\{x\in[0,1)\colon |T^{n}_{\beta}x-x|<\varphi(n) \textrm{ for infinitely many } n\in\mathbb{N}\}.
\end{equation*}
Applying Boshernitzan's outstanding results about the quantitative recurrence problem in a measure dynamical system \cite{Bo} to the beta-expansion $([0,1],T_{\beta})$, one knows that for $\mathcal{L}$-almost every $x\in [0,1)$,
  \begin{equation*}
    \liminf_{n\rightarrow\infty}n|T^{n}_{\beta}x-x|<+\infty.
  \end{equation*}
Recently, Hussain, Li, Simmons and Wang \cite{HuLiSiWa} showed that $$
\mathcal{L}(\mathfrak{R}_{\beta}(\varphi))=0, \ {\text{or}}\ 1 \Longleftrightarrow \sum_{n=1}^{\infty}\varphi(n)<\infty,\ {\text{or}}\ =\infty,
$$
where the exponentially mixing property of the beta-expansion is essential to their argument.

With the same idea used in the proof of Theorem \ref{T:main}, we can go a little further and we do not need the exponentially mixing property.

Let $L\colon [0,1]\rightarrow[0,1]$ be a Lipschitz function and let
\begin{equation*}
    \mathfrak{R}_{\beta}(L,\varphi)=\{x\in[0,1)\colon |T^{n}_{\beta}x-L(x)|<\varphi(n) \textrm{ for infinitely many } n\in\mathbb{N}\}.
\end{equation*}
Rewrite $\mathfrak{R}_{\beta}(L,\varphi)$ to express its limsup nature:
\begin{equation*}
    \mathfrak{R}_{\beta}(L,\varphi)=\bigcap^{\infty}_{m=1}\bigcup^{\infty}_{n=m}\bigcup_{w\in \Sigma_n(\beta)}\mathfrak{I}(w;L,\varphi),
\end{equation*}
where $$
\mathfrak{I}(w;L,\varphi)=\Big\{x\in \mathfrak{I}(w): |T^nx-L(x)|<\varphi(n)\Big\}.
$$
Let $\kappa>0$ be a Lipschitz constant of $L(x)$, i.e., for any $x,y\in[0,1]$, we have $|L(x)-L(y)|\leq \kappa|x-y|$.
In analogy with Lemma \ref{L:length} and Proposition \ref{P:Iwp/Iw}, we have
\begin{lemma}
  Let $\beta>1$. For any $n\in \mathbb{N}$ with $n>\log_{\beta} (3\kappa)$ and $w\in \Sigma_{n}(\beta)$, the set $\mathfrak{I}(w;L,\varphi)$
  is an interval and
\begin{align*}
\left\{
                             \begin{array}{ll}
                               |\mathfrak{I}(w;L,\varphi)|\leq 3\varphi(n)/\beta^{n}, & \hbox{for $w\in \Sigma_n(\beta)$;} \\
                               |\mathfrak{I}(w;L,\varphi)|\ge \frac{1}{4}\varphi(n)/\beta^{n}, & \hbox{for $w\in \Xi_n(\beta)$.}
                             \end{array}
                           \right.
\end{align*}
\end{lemma}
\begin{proof} For any $w\in \Sigma_n(\beta)$, we define
  \begin{equation*}
    \mathfrak{f}_{w}(x)=\beta^{n}\left(x-\sum^{n}_{i=1}\frac{w_{i}}{\beta^{i}}\right)-L(x)=T^n_{\beta}(x)-L(x), \ x\in \mathfrak{I}(w).
  \end{equation*}
It is easy to see that for $n>\log_{\beta} (3\kappa)$, we have
$$
\frac{2\beta^n}{3}<\frac{\mathfrak{f}_{w}(x)-\mathfrak{f}_{w}(y)}{x-y}<\frac{4\beta^n}{3} \ {\text{for}}\ x,y\in \mathfrak{I}(w),
$$
and thus $\mathfrak{f}_{w}(x)$ is continuous and strictly increasing. So, $\mathfrak{I}(w;L,\varphi)$ is an interval.

For any $s,t\in \mathfrak{I}(w;L,\varphi)$ with $s< t$,
$$\frac{2\beta^n}{3}<\frac{\mathfrak{f}_{w}(t)-\mathfrak{f}_{w}(s)}{t-s}<\frac{2\varphi(n)}{t-s}, $$
so $$|\mathfrak{I}(w;L,\varphi)|\leq 3\varphi(n)/\beta^{n}.$$

If $w\in \Xi_{n}(\beta)$, i.e., $|\mathfrak{I}(w)|=\beta^{-n}$,  write
\begin{equation*}
  \mathfrak{I}(w)=\left[\sum^{n}_{i=1}\frac{w_{i}}{\beta^{i}}, \sum^{n}_{i=1}\frac{w_{i}}{\beta^{i}}+\frac{1}{\beta^{n}}\right):=[a, a+\beta^{-n}).
\end{equation*}
Since $|L(a+\beta^{-n})-L(a)|\le \kappa \beta^{-n}<1/3$, one can see that the interval
\begin{align*}
 \mathfrak{f}_{w}(\mathfrak{I}(w))=\Big[-L(a),\ 1-L(a+\beta^{-n})\Big)
\end{align*}
is of length larger than $2/3$. Note that $-L(a)\leq 0 \leq 1-L(a+\beta^{-n})$.
Thus the interval
$$
\mathfrak{f}_{w}(\mathfrak{I}(w; L, \varphi))=\mathfrak{f}_{w}(\mathfrak{I}(w))\cap \Big(-\varphi(n), \varphi(n)\Big)
$$
is of length larger than $$\min\Big\{1/3, \varphi(n)\Big\}\ge \frac{\varphi(n)}{3}.$$ Therefore, there exists $s, t\in \mathfrak{I}(w;L,\varphi)$ such that $$
\frac{\varphi(n)}{3}=|\mathfrak{f}_{w}(t)-\mathfrak{f}_{w}(s)|\le \frac{4\beta^n}{3}\cdot |t-s|,
$$
so, $|\mathfrak{I}(w;L,\varphi)|\geq \frac{1}{4}\varphi(n)/\beta^{n}$.
\end{proof}
\begin{proof}[Sketch of the proof of Theorem \ref{T:rec}]
The convergent part is direct, since for any $n\in \mathbb{N}$ such that $n>\log_{\beta}(3\kappa)$,
\begin{equation*}
  \sum_{w\in \Sigma_{n}(\beta)}|\mathfrak{I}(w;L,\varphi)|\leq \#\Sigma_{n}(\beta) \cdot\frac{3\varphi(n)}{\beta^{n}}
  \leq \frac{\beta^{n+1}}{\beta-1}\cdot\frac{3\varphi(n)}{\beta^{n}}=\frac{3\beta}{\beta-1}\varphi(n).
\end{equation*}
Then the Borel-Cantelli lemma applies.

For the divergence part, we first prove that there exists a constant $\rho>0$ depending only on $\beta$ such that for any $m\in \mathbb{N}$ and $\tau\in \Xi_{m}(\beta)$, one has
\begin{equation*}
    \mathcal{L}(\mathfrak{R}_{\beta}(L,\varphi)\cap \mathfrak{I}(\tau))\geq \rho |\mathfrak{I}(\tau)|.
\end{equation*}
Suppose that $\lambda\in(0,1)$ is a real number satisfying the condition in Proposition \ref{P:numperf}. Let $c$ be a positive integer large enough such that $c>\max\{-\log_{\beta}\lambda,\log_{\beta}(3\kappa)\}$ and for all $n\geq c$, we have $\lambda\beta^{n}\geq 4\beta/(\beta-1)$.

Fix $m\in \mathbb{N}$ and $\tau\in \Xi_{m}(\beta)$. For any $n\geq m$, let
\begin{equation*}
    \widetilde{\mathbb{A}}_{n}(\tau)=\{w\in \Xi_{n}(\beta)\colon w|_{m}=\tau\}.
\end{equation*}
Given $l\geq m+c$, for any $m\leq n\leq l-c$ and $w\in\widetilde{\mathbb{A}}_{n}(\tau)$, let
\begin{equation*}
    \widetilde{\mathbb{H}}_{l}(w)=\{v\in \widetilde{\mathbb{A}}_{l}(\tau)\colon \mathfrak{I}(w;L,\varphi)\cap \mathfrak{I}(v;L,\varphi)\neq \varnothing\}.
\end{equation*}
Define $\widetilde{\mathbb{H}}^{\ast}_{l}(w)$ as before to be a set consisting of the lexicographically smallest and the lexicographically largest words in $\widetilde{\mathbb{H}}_{l}(w)$. For any $l\geq m+c$, let
\begin{equation*}
     \widetilde{\mathbb{B}}_{l}(\tau)=\widetilde{\mathbb{A}}_{l}(\tau)\backslash \bigcup_{m\leq n\leq l-c}\bigcup_{w\in \widetilde{\mathbb{A}}_{n}(\tau)}\widetilde{\mathbb{H}}^{\ast}_{l}(w)
\end{equation*}
and define
\begin{equation*}
    \mathfrak{F}_{l}(\tau;L,\varphi)=\bigcup_{v\in \widetilde{\mathbb{B}}_{l}(\tau)}\mathfrak{I}(v;L,\varphi).
\end{equation*}
For the position relations between the sets $[0,1), \mathfrak{I}(\tau), \mathfrak{I}(w)$ and $\mathfrak{I}(v)$, we have the following diagram:
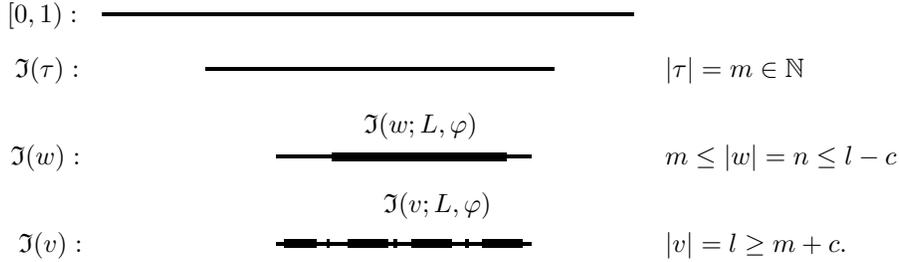
\begin{center}
  \unitlength 3pt
\begin{picture}(115,37)
\linethickness{1pt}
  \put(12,33){\line(1,0){67}}
  \put(0,32){$[0,1):$}

  \put(25,26){\line(1,0){44}}
  \put(1,25){$\mathfrak{I}(\tau):$}
  \put(83,25){$|\tau|=m\in\mathbb{N}$}

  \put(34,15){\line(1,0){32}}
  \put(0.5,14){$\mathfrak{I}(w):$}
   \put(45,18){$\mathfrak{I}(w;L,\varphi)$}
  \put(83,14){$m\leq |w|=n\leq l-c$}

  \put(34,4){\line(1,0){32}}
  \put(1.5,3){$\mathfrak{I}(v):$}
  \put(47.5,8){$\mathfrak{I}(v;L,\varphi)$}
  \put(83,3){$|v|=l\geq m+c.$}

    \put(40.5,3.5){\line(0,1){1}}
    \put(49,3.5){\line(0,1){1}}
    \put(58,3.5){\line(0,1){1}}

\linethickness{3pt}
  \put(41,15){\line(1,0){22}}
  \put(35,4){\line(1,0){4}}
  \put(43,4){\line(1,0){5}}
  \put(51,4){\line(1,0){5}}
  \put(60,4){\line(1,0){5}}
\end{picture}
\end{center}

Using item $(iii)$ in Lemma \ref{L:full}: we have $\tau \varsigma\in \widetilde{\mathbb{A}}_{l}(\tau)$ for all $\varsigma\in \Xi_{l-m}(\beta)$; and for each $w\in \widetilde{\mathbb{A}}_{n}(\tau)$, we have $wv\in \widetilde{\mathbb{A}}_{l}(\tau)$ for all $v\in \Xi_{l-n}(\beta)$. Then, by Proposition \ref{P:numperf}, it follows that \begin{align*}&\#\widetilde{\mathbb{A}}_{l}(\tau)\geq \#\Xi_{l-m}(\beta)\geq \lambda \#\Sigma_{l-m}(\beta)\geq \lambda \beta^{l-m}, \ {\text{for}}\ l\ge m+c;\\
&\#\widetilde{\mathbb{A}}_{l}(\tau)\geq \#\Xi_{l-n}(\beta)\cdot\#\widetilde{\mathbb{A}}_{n}(\tau)\geq \lambda\beta^{l-n}\#\widetilde{\mathbb{A}}_{n}(\tau), \ {\text{for}}\ l\ge n+c.\end{align*}
Thus by the choice of $c$,
\begin{align*}
  \# \widetilde{\mathbb{B}}_{l}(\tau)&\ge \#\widetilde{\mathbb{A}}_{l}(\tau)-2\sum_{n=m}^{l-c}\# \widetilde{\mathbb{A}}_{n}(\tau)
\ge \#\widetilde{\mathbb{A}}_{l}(\tau)-2\#\widetilde{\mathbb{A}}_{l}(\tau)\lambda^{-1}\sum_{n=m}^{l-c}\beta^{n-l}\\
&\ge \#\widetilde{\mathbb{A}}_{l}(\tau)-\#\widetilde{\mathbb{A}}_{l}(\tau)\frac{2\beta^{-c+1}}{\lambda(\beta-1)}\ge \frac{1}{2}\cdot \#\widetilde{\mathbb{A}}_{l}(\tau).
\end{align*}
Consequently, we have \begin{equation*}
  \mathcal{L}(\mathfrak{F}_{l}(\tau;L,\varphi))\geq \#\widetilde{\mathbb{B}}_{l}(\tau)\cdot \frac{\varphi(l)}{4\beta^{l}}\geq \frac{1}{2}\#\widetilde{\mathbb{A}}_{l}(\tau)\cdot \frac{\varphi(l)}{4\beta^{l}}
  \geq\frac{\lambda\varphi(l)}{8\beta^{m}}=\frac{\lambda}{8}\varphi(l)|\mathfrak{I}(\tau)|.
\end{equation*}

As in the proof of Proposition \ref{P:leq}, we can show that for any $n\geq m+c$ and $l\geq n+c$,
\begin{align*}
  \mathcal{L}(\mathfrak{F}_{n}(\tau;L,\varphi)\cap \mathfrak{F}_{l}(\tau;L,\varphi))&\leq 3\varphi(l)\mathcal{L}(\mathfrak{F}_{n}(\tau;L,\varphi))\\
  &\leq \frac{24}{\lambda|\mathfrak{I}(\tau)|}\mathcal{L}(\mathfrak{F}_{n}(\tau;L,\varphi))\mathcal{L}(\mathfrak{F}_{l}(\tau;L,\varphi)).
\end{align*}
Then the Chung-Erd\"{o}s inequality enables us to conclude that
$$\mathcal{L}(\mathfrak{R}_{\beta}(L,\varphi)\cap \mathfrak{I}(\tau))\geq \rho|\mathfrak{I}(\tau)|,$$ where $\rho:=\frac{\lambda}{24}>0$ is a constant depending only on $\beta$.

Now, we show that for any $m\in \mathbb{N}$ and $u\in \Sigma_{m}(\beta)$, one has $$\mathcal{L}(\mathfrak{R}_{\beta}(L,\varphi)\cap \mathfrak{I}(u))\geq \frac{\rho}{\beta} \cdot |\mathfrak{I}(u)|.$$

Let $m\in \mathbb{N}$ and $u\in \Sigma_{m}(\beta)$. Let  $k_{0}\geq 0$ be the smallest integer such that $u 0^{k_{0}}1\in \Sigma_{m+k_{0}+1}(\beta)$.
Then by item $(i)$ in Lemma \ref{L:admin} for the admissibility of a word, we have
$$\mathfrak{I}(u)=\mathfrak{I}(u 0^{k_{0}}) \ {\text{and}}\ u 0^{k_{0}+1}\in \Xi_{m+k_{0}+1}(\beta).$$ Thus,
\begin{equation*}
    |\mathfrak{I}(u 0^{k_{0}+1})|=\beta^{-m-k_{0}-1}\geq |\mathfrak{I}(u 0^{k_{0}})|/\beta=|\mathfrak{I}(u)|/\beta.
\end{equation*}
Therefore,
\begin{equation*}
    \mathcal{L}(\mathfrak{R}_{\beta}(L,\varphi)\cap \mathfrak{I}(u))\geq \mathcal{L}(\mathfrak{R}_{\beta}(L,\varphi)\cap \mathfrak{I}(u0^{k_{0}+1}))\geq \rho|\mathfrak{I}(u0^{k_{0}+1})|\geq \frac{\rho}{\beta}\cdot |\mathfrak{I}(u)|.
\end{equation*}
\end{proof}

\end{document}